\documentclass{amsart}

\usepackage{amssymb}
\usepackage{amsmath}
\usepackage{graphics}
\usepackage{enumerate}
\usepackage{xcolor}
\usepackage{dsfont}
\usepackage{hyperref}
\usepackage[normalem]{ulem} 
\usepackage{tikz}
\usepackage{tikz-3dplot}
\usetikzlibrary{decorations.markings,decorations.pathmorphing}
\usepackage{pgfplots}

\usepackage{graphicx}
\usepackage{subcaption}

\usepackage{latexsym}

\usepackage{dsfont}

\def\bb#1\eb{\textcolor{blue}{#1}} 
\def\br#1\er{\textcolor{red}{#1}} %
\def\bm#1\em{\textcolor{purple}{#1}} %

\newtheorem{theorem}{Theorem}[section]
\newtheorem{proposition}[theorem]{Proposition}

\newtheorem{lemma}[theorem]{Lemma}

\theoremstyle{definition}
\newtheorem{definition}[theorem]{Definition}

\theoremstyle{remark}
\newtheorem{remark}[theorem]{Remark}

\numberwithin{equation}{section}

\newcommand{\F}{\ensuremath{\mathcal{F}}}

\newcommand{\rank}{\ensuremath{\mathrm{rank}\ }}

\newcommand{\RR}{\mathds R}
\newcommand{\N}{\mathds N}

\newcommand{\selfadjointW}{\ensuremath{\mathtt{W} }}

\newcommand{\selfadjointV}{\ensuremath{\mathtt{V}} }

\newcommand{\WilkingH}{\ensuremath{\mathcal{H} } }

\newcommand{\WilkingV}{\ensuremath{\mathcal{V} } }

\newcommand{\h}{\mathfrak{h}}

\newcommand{\vp}{\mathfrak{v}}

\newcommand{\Dh}{D^{\mathfrak{h}}}

\begin{document}


\title[Equifocal  Finsler submanifolds and analytic maps ]{On equifocal  Finsler submanifolds and analytic maps }

\author[M.M. Alexandrino]{Marcos M. Alexandrino}

\author[B. Alves]{Benigno Alves}

\author[M. A. Javaloyes]{Miguel Angel Javaloyes }


\address{Benigno O. Alves \hfill\break\indent 
	Instituto de Matem\'{a}tica e Estat\'{\i}stica\\
	Universidade Federal da Bahia, \hfill\break\indent
	Rua Bar\~ao de Jeremoabo ,40170-115 Salvador, Bahia, Brazil}
\email{benignoalves@ufba.br;gguialves@hotmail.com}

\address{Marcos M. Alexandrino \hfill\break\indent 
Instituto de Matem\'{a}tica e Estat\'{\i}stica\\
Universidade de S\~{a}o Paulo, \hfill\break\indent
 Rua do Mat\~{a}o 1010,05508 090 S\~{a}o Paulo, Brazil}
\email{marcosmalex@yahoo.de, malex@ime.usp.br}

\address{Miguel \'Angel Javaloyes \hfill\break\indent
Departamento de Matem\'aticas, Universidad de Murcia, \hfill\break\indent Campus de Espinardo,
\hfill\break\indent 30100 Espinardo, Murcia, Spain} 
\email{majava@um.es}

\thanks{The first and second authors were supported by  Funda\c{c}\~{a}o de Amparo a Pesquisa do Estado de S\~{a}o Paulo-FAPESP (Tematicos: 2016/23746-6). The second author was supported by CNPq (PhD fellowship) and partially
supported by  PDSE-Capes (PhD sandwich program). This work is a result of the activity developed within the framework of the Programme in Support of Excellence Groups of the Regi\'on de Murcia, Spain, by Fundaci\'on S\'eneca, Science and Technology Agency of the Regi\'on de Murcia. The third author was partially supported by MICINN/FEDER project reference
PGC2018-097046-B-I00 and Fundaci\'on S\'eneca  (Regi\' on de Murcia) project reference 19901/GERM/15, Spain}






\begin{abstract}
  A relevant property of equifocal submanifolds is that their parallel sets are still immersed submanifolds, which makes them a 
 natural generalization of the so-called isoparametric submanifolds. In this paper, we prove that the regular fibers of an analytic map 
 $\pi:M^{m+k}\to B^{k}$ are equifocal whenever  $M^{m+k}$ is endowed with a complete Finsler metric and there is a restriction of
 $\pi$ which is a Finsler submersion for a certain Finsler metric on the image. In addition, we prove that when the fibers provide a singular
 foliation on $M^{m+k}$, then this foliation is Finsler.

\end{abstract}


\maketitle

\section{Introduction}

Roughly speaking, as the name  might  suggest, 
an equifocal submanifold $L$  in a Riemannian manifold $M$
is  one whose \emph{``parallel sets"} are always (immersed) submanifolds, 
even when they contain focal points of $L$ (see the formal definition below).

From Somigliana's article on Geometric Optics in $1919$,
and Segre and Cartan's work in the $1930s$ 
to the present day, this natural concept
received different names and was treated in
different ways, but it has been playing an 
important role in the theory of submanifolds 
and problems with symmetries; for surveys on the relation between 
the theory of isoparametric submanifolds and isometric actions 
see Thorbergsson's surveys \cite{Th,ThSurvey2,ThSurvey3}
and for books on these topics   
see e.g., \cite{AlexBettiol,BerndtConsoleOlmos,PalaisTerng}.

Isoparametric submanifolds 
and regular orbits of isometric actions
are examples of equifocal submanifolds. Both kinds of submanifolds 
are leaves of the 
so-called singular Riemannian foliations (SRF for short), 
i.e., singular foliations whose leaves are locally equidistant, see e.g., \cite{AlexRafaelDirk}. 
On the other hand, the regular leaves of an SRF are equifocal (see \cite{AlexToeben2}).
Therefore, so far, equifocality and SRF's are closely related.

Also over the past decades, important examples have been
presented as a pre-image of  regular values of
analytic maps, see \cite{FerusKarcherMunzner} and \cite{RadeschiClifford}.
And there are good reasons for that. For example, as proved in \cite{LytchakRadeschi}, 
the leaves of SRF's with closed leaves in Euclidean spaces
are always pre-images of polynomial maps.

Several problems related to
symmetries and wavefronts presented in Riemannian
geometry are also natural problems in Finsler geometry.
For example,  the Finsler distance
(a particular case of transnormal functions)
has been used to model forest wildfires
and Huygens' principle, see \cite{Markvorsen}. This is one of the several motivations for
a recent development of Finsler's isoparametric theory,
and again isoparametric and transnormal analytic functions
naturally appear; see \cite{AleAlvesDehkordi,DongPeilongQun,He-Yin-Shen,HeQunSongTingYiBing,HeQuHe,HeQunPeilongDongSongtingYin,MingXu}.
For other applications see  \cite{Eikonal} and \cite{HengamegSaa}.

Very recently we  have  laid  in  \cite{Alexandrino-Alves-Javaloyes}  the foundations of the new
theory of \emph{singular Finsler foliations} (SFF), i.e., a singular foliation
where the leaves are locally equidistant 
(but the distance from
a plaque or leaf $L_a$ to $L_b$ does not need to be equal
to the distance  from a plaque or leaf $L_b$ to $L_a$), see Lemma \ref{lemma-equidistant}.
The partition of a Finsler manifold into 
orbits of a Finsler action is a natural example of SFF.
We have also presented non-homogeneous examples using analytic maps \cite[Example 2.13]{Alexandrino-Alves-Javaloyes}. 
Among other results, we proved the equifocality of the regular  leaves
of an  SFF on a Randers manifold, where the wind  is an infinitesimal homothety. 
Nevertheless, to prove the equifocality of the regular leaves of an SFF  in the general case
becomes a challenging problem, presenting new aspects
that did not appear in the analogous problem in Riemannian geometry.

For all these reasons, it is natural to ask if we can
prove the equifocality of  leaves of  SFF's  described
by analytic maps. In this article, we tackle this issue proving it  under natural conditions. During our investigation
we generalize some Riemannian techniques (like  Wilking's distribution)
and  obtain some results that,  as far as we know, do not appear in the literature even in the  Riemannian case.

Throughout this article, $(M,F)$ will always be a complete analytical Finsler  manifold.

In order to formalize the concept of ``parallel sets'' and equifocality,
let us start by recalling the definition of equifocal hypersurfaces.

Let  $L$ be an  (oriented) hypersurface with  normal unit vector field $\xi$  (in the non-reversible Finsler case, there could be two independent normal unit vector fields to $L$).  
We  define \emph{the endpoint map} $\eta_{\xi}: L\to M$
as $\eta_{\xi}^{r}(p)=\gamma_{\xi_p}(r)$, where $\gamma_{\xi_p}$ is the 
unique geodesic with $\dot\gamma_{\xi_p}(0)=\xi_p$.
In this case  the parallel sets are $L_r=\eta_{\xi}^{r}(L)$. 
\emph{The hypersurface $L$ is said  equifocal  if $d \eta_{\xi}^{r}$ has constant
	rank for all $r.$ } From  the constant rank theorem,  it is not difficult to check that if $L$ is compact, then
the parallel sets $L_r$ are immersed submanifolds (with possible intersections). 

If the codimension of $L$ is greater than 1,  we need to clarify which normal unit vector field $\xi$ along $L$ 
we consider as there are many different choices.
We will follow the same approach as in the Riemannian case.

When $L$ is a regular leaf of an SFF denoted by $\F$
(i.e., $L$ has maximal dimension),
there exists a neighborhood $U$ (in $M$) of the point $p\in L$ and a 
Finsler submersion $\rho: U\to S$ so that $\F_U=\F\cap U$ are the
pre-images of the map $\rho$. 
In this case, $\xi$ can be chosen  as a  projectable  
normal  vector field along $L\cap U$  
with respect to $\rho$, i.e.,  an $\F$-\emph{local basic vector field}.
Now, we say that \emph{$L$ is an equifocal submanifold if, for each local $\F$-basic
	unit vector field $\xi$ along   $L$, the differential map
	$d \eta_{\xi}^{r}$ has constant rank for all $r$}.

Finally when $\pi: M\to B$ is an analytic map, $c$ is a regular value and  $L=\pi^{-1}(c)$, 
$\xi$ can be chosen  as a basic vector field along $L$, i.e.,  \emph{$\pi$-basic}.  Once we want to consider
``parallel sets'' of codimension bigger than one, it is  natural to expect that at least  near $L$ the fibers are equidistant.
 Therefore, in this case  \emph{  we will demand that there exists a saturated neighborhood $\widetilde U$ of $L$ where $\pi$ is regular
and the fibers of $\pi$ restricted to $\widetilde U$ are the leaves of a regular Finsler foliation.
In other words, the restriction $\pi_{\widetilde U}:\widetilde U\to\pi(\widetilde U)\subset B$ is a Finsler submersion
for some Finsler metric on $\pi(\widetilde U)$.  
	With this assumption,  we will say that $L=\pi^{-1}(c)$ is an equifocal submanifold if
	for each $\pi$-basic  unit vector field $\xi$ along $L$, the differential map
	$d \eta_{\xi}^{r}$ has constant rank for all $r$.} 
	 Observe that in this definition, the partition $\F_\pi=\{\pi^{-1}(c) \}_{c\in B}$ of $M$ by the fibers of $\pi$ is not required  to be a singular foliation.

\begin{theorem}
	\label{SubAnal}
	Let $(M,F)$ be an analytic connected complete Finsler manifold and 
	$\pi:M^{m+k}\rightarrow B^{k}$  an analytic map between  analytic manifolds. Assume that
	the pre-image $\pi^{-1}(d)$ is always a connected set and that there exists  an
	open subset $U$ where $\pi$ is regular and such that $\F_\pi|_U$ is a Finsler foliation. 
	We have that
	\begin{itemize}
		\item[(a)] if $q\in \pi^{-1}(c)$ is a regular point (i.e., $d \pi_q$ is surjetive), then all points of $\pi^{-1}(c)$ are regular, i.e.,
		$c$ is a regular value,  and  the restriction of $\F_\pi$ to the set of regular points of $\pi$ is a Finsler foliation. 
		\item[(b)] Each regular level set $\pi^{-1}(c)$ is an equifocal submanifold. In addition, the image of the
		endpoint map $\eta_{\xi}^{r}:\pi^{-1}(c)\to M$ (for $r$ and $\xi$ fixed) is contained in a level set. 
		\item[(c)] If  $\mathcal F_{\pi}$ 
		is a smooth singular foliation, then $\mathcal F_{\pi}$ is a singular Finsler foliation. 
	\end{itemize}
	
\end{theorem}

\begin{remark}
	Item (a)  shows how the hypothesis influences the singular fibers. 
	In fact, consider $f:\mathbb{R}^{2}\to\mathbb{R}$ defined as $f(x_1,x_2)=x_1x_2$. Note that
	$f^{-1}(0)$ is the stratified set $\{x_2=0\}\cup\{x_1=0\}$. Also note that each point  $ p\in f^{-1}(0)$ different
	from  $(0,0)$  is a regular point, i.e., $ \nabla f(p)\neq (0,0)$. In other words, 
	in this very natural and simple example item (a) is not fulfilled.
\end{remark}

\begin{remark}
	As far as we know, the theorem is  new even in the Riemannian case.
	In fact, in \cite{Alexandrino-IRSS}  the first author dealt with transnormal analytic maps,
	assuming also that the normal bundle of the fibers was an integrable distribution
	(on the set of regular points).
\end{remark}

There are a few interesting open questions. The first one 
is under which hypotheses can one ensure that the singular set is not
just a stratification, but a submanifold. 
This problem was solved  in \cite{AleAlvesDehkordi}, 
in the case where the map was an analytic  function $f:M\to\mathbb{R}.$
The second question is even when  the singular sets are submanifolds, and 
the fibers are equidistant, under which conditions  $\mathcal F_{\pi}=\{\pi^{-1}(c) \}_{c\in B}$  will be a smooth
singular foliation, i.e.,   for each $p\in  \pi^{-1}(d)$ and $v\in T_p(\pi^{-1}(d))$,   there 
is a smooth vector field $X$ tangent to the fibers with $X_p=v$.

This paper is organized as follows. In \S \ref{section-preliminaries}, we briefly review a few facts
on Finsler geometry and Finsler submersions. In \S \ref{section-wilkingsdistribution}, we review the construction
of Wilkings' distribution in the Finsler case. In \S \ref{section-singular-analytic-submersion},
we present the proof of Theorem \ref{SubAnal}. 
In \S \ref{apendice},  we discuss the definition of the curvature tensor (related to 
 the definition of Jacobi tensor presented in \S \ref{section-preliminaries})
 and some proofs of results (presented 
in \S \ref{section-wilkingsdistribution}) on  Jacobi fields. In this way, we hope to make the paper accessible also to readers without previous training in Finsler geometry, or who are more interested in understanding the Riemannian case of the theorem. 

\

\emph{Acknowledgements}  We are very grateful to Prof.~Daniel Tausk (IME-USP, S\~ao Paulo, Brasil) for explaining some aspects of the theory of Morse-Sturm systems to us.


\section{Preliminaries}
\label{section-preliminaries}

In this section we fix some notations and briefly review a few facts about Finsler Geometry
that will be used in this paper. For more details see \cite{BCS,Sh01,Zhongmin-Shen}. 

\subsection{Finsler metrics }
\label{section-finsler-metrics}

Let $M$ be a manifold. %
We say that a continuous function $F:TM\to [0,+\infty)$ is a \emph{Finsler metric} if
\begin{enumerate}
\item $F$ is smooth on $TM \setminus \bf 0$,
\item $F$ is positive homogeneous of degree $1$, that is, $F(\lambda v)=\lambda F(v)$ for
every $v\in TM$ and $\lambda>0$,
\item for every $p\in M$ and $v\in T_pM \setminus \{0\}$, the \emph{fundamental tensor} of $F$ defined as
$$ g_{v}(u,w)=\frac{1}{2}\frac{\partial^{2}}{\partial t\partial  s} F^{2}(v+tu+sw)|_{t=s=0}$$
for any $u,w\in T_{p}M$ is a nondegenerate positive-definite bilinear  symmetric form. 
\end{enumerate} 
In particular,  if $M$ is a vector space $V$, and $F:V\to \RR$ is a function smooth on $V\setminus \{\bf 0\}$  and satisfying the properties  $(2)$ and $(3)$ above, 
then $(V,F)$ is called a \emph{Minkowski space}.

Recall that the \emph{Cartan tensor} associated with the Finsler metric $F$ is defined as
$$C_{v}(w_{1},w_{2},w_{3}):= \frac{1}{4} \frac{\partial^{3}}{\partial s_{3}\partial s_{2}\partial s_{1}} F^{2}(v+ \sum_{i=1}^{3} s_{i}w_{i})|_{s_{1}=s_{2}=s_{3}=0}$$
for every $p\in M$,  $v\in T_pM\setminus \bf 0$,  and  $w_{1}, w_{2}, w_{3}\in T_{p}M.$


 \subsection{The Chern connection and the induced covariant derivative} 
In  \S \ref{apendice1}, we will 
discuss the concept of anisotropic connection,  which means that the  Christoffel symbols are functions on the 
tangent bundle. 
Here we just  briefly review the concept of 
 Chern connection associated with a Finsler metric  as a family of affine connections $\nabla^V$, with  $V$ a vector field without 
 singularites on an open subset $\Omega$ (see \cite[Eq. (7.20) and (7.21)]{Sh01} and also \cite{J13}):
\begin{enumerate}
\item $\nabla^{V}$ is \emph{torsion-free}, namely,
$$ \nabla^{V}_{X}Y-\nabla^{V}_{Y}X=[X,Y]$$
for every vector field $X$ and $Y$ on $\Omega$,
\item $\nabla^{V}$ is \emph{almost g-compatible}, namely,
$$X\cdot g_{V}(Y,Z)=g_{V}(\nabla^{V}_{X}Y,Z)+g_{V}(Y,\nabla^{V}_{X}Z)+2C_{V}(\nabla^{V}_{X}V,Y,Z), $$
where $X$, $Y$, and $Z$ are vector fields on open set $\Omega$ and $g_V$ and $C_V$ are the tensors on $\Omega$ such that $(g_{V})_p=g_{V_p}$ and $(C_{V})_p=C_{V_p}$.
\end{enumerate}
 It can be checked that the Christoffel symbols of $\nabla^V$ only depend on $V_p$ at every $p\in M$, and not on the particular extension of $V$. Therefore, the Chern connection is an anisotropic connection. 
 Moreover, it   is positively homogeneous of degree zero, namely, $\nabla^{\lambda v}=\nabla^v$ for all $v\in A$ and $\lambda>0$. 
For an explicit expression of the Christoffel symbols of the Chern connection in terms of the coefficients of the fundamental and the Cartan tensors 
see \cite[Eq. (2.4.9)]{BCS}. 
 The Chern connection provides an (anisotropic) curvature tensor $R$, which,  for every $p\in M$ and $v\in T_pM\setminus \{\bf 0\}$, determines a linear map 
$R_v:T_pM\times T_pM\times T_pM\rightarrow T_pM$. The Chern curvature tensor can be defined using the anisotropic calculus as in Eq. \eqref{curten}, or in coordinates, with the help of the non-linear connection (see \cite[Eq. (3.3.2)]{Bao-Robles-Shen}). 
Given a smooth curve $\gamma: I\subset \mathbb{R}\to M$ and a  smooth vector field $W\in \mathfrak{X}(\gamma)$ along $\gamma$ without singularites, where $\mathfrak{X}(\gamma)$ denotes the smooth sections of the pullback bundle $\gamma^{*}(TM)$ over $I$, the Chern connection induces a {\em covariant derivative}
$D^{W}_\gamma$ along $\gamma$, such that $(D^{W}_\gamma X) (t)=\nabla^{W(t)}_{\dot\gamma(t)}X$, with the identification $X(t)=X_{\gamma(t)}$, when $X\in {\mathfrak X}(M)$.
If the image of the curve is contained in a chart $(\Omega,\varphi)$ 
(using the Christoffel symbols introduced in Eq. \eqref{chrissymbols}) 
the covariant derivative is expressed as
\[(D^{W}_\gamma X) (t)=\sum_{k=1}^n(\dot X^k(t)+\sum_{i,j=1}^nX^i(t)\dot\gamma^j(t)\Gamma^k_{\,\, ij}(W(t)))\partial_k,\]
where $X,W\in {\mathfrak X}(\gamma)$ and $X^l(t)$ and $\dot\gamma^l(t)$ are the coordinates of $X(t)$ and $\dot\gamma(t)$, respectively, for $t\in I$.
  When 
$\dot\gamma(t)\not=0$ for all $t\in I$, we can take as a reference vector $W=\dot\gamma$. In such a case, we will use the notation $X':=D^{\dot\gamma}_\gamma X$ whenever there is no possible confusion about $\gamma$.

%
\subsection{Geodesics and Jacobi fields}
We will say that a smooth curve $\gamma:I\subset \RR\rightarrow M$ is a geodesic of $(M,F)$ if it is an auto-parallel curve of the covariant derivative induced by the Chern connection, namely, $D^{\dot\gamma}_\gamma \dot\gamma=0$. 
Given a vector $v\in TM\setminus \bf 0$, there is a unique geodesic $\gamma_v$ such that $\dot\gamma_v(0)=v$. 

 When we consider a geodesic variation of $\gamma$, it turns out that the variational vector field $J$ is characterized by solving the differential equation
\begin{equation}\label{jacobi}
J'' (t)+ R_{\dot{\gamma}(t)}(J(t))=0, 
\end{equation}
where, for every $p\in M$ and $v\in T_pM\setminus \{\bf 0\}$, we define the operator $R_v:T_pM\rightarrow T_pM$ as $R_v(w)=R_v(w,v)v$ for
every $w\in T_pM$ (see for example \cite[Prop. 2.11]{J20}). Given a geodesic $\gamma$, the operator $R_{\dot\gamma}$,  defined for vector fields along $\gamma$, and the solutions of Eq. \eqref{jacobi} are known, respectively, as the {\em Jacobi operator of $\gamma$}, and {\em the Jacobi fields of $\gamma$. }

 When we consider a submanifold $L$, the minimizers from $L$ are orthogonal geodesics (see \cite{AlJav19}), and variations
of orthogonal geodesics will be given by $L$-Jacobi fields, which will be introduced below. First let us define orthogonal vectors. Given a submanifold $L\subset M$, we say that a vector $v\in T_pM$, with $p\in L$, is orthogonal to $L$ if $g_v(v,u)=0$ for all $u\in T_pL$. The subset of orthogonal vectors to $L$ at $p\in L$, denoted by $\nu_pL$, could  be non-linear, but it is a smooth cone (see \cite[Lemma 3.3]{JavSoa15}). We say that a geodesic $\gamma:I\rightarrow M$ is orthogonal to $L$ at $t_0\in I$ if $\dot\gamma(t_0)$ is orthogonal to $L$.
\begin{definition}
\label{definition-eq-L-Jacobi}
Let $L$ be a submanifold of a Finsler manifold $(M,F)$ and $\gamma:[a,b)\rightarrow M$ a geodesic orthogonal to $L$ at $p=\gamma(a)$.
We say that a Jacobi field is \emph{$L$-Jacobi} if 
\begin{itemize}
\item $J(a)$ is tangent to $L$, 
\item  $\mathcal{S}_{\dot{\gamma}(a)}J(a)=\tan_{\dot{\gamma}(a)} J'(a)$,
\end{itemize}
where  
$\mathcal{S}_{\dot{\gamma}}:T_{p}L\to T_{p}L$ is the \emph{shape operator} 
defined as $\mathcal{S}_{\dot{\gamma}}(u)=\mathrm{tan}_{\dot\gamma(a)} \nabla^{\dot\gamma(a)}_{u} \xi $, with $\xi$
 an orthogonal vector field along $L$ such that $\xi_p=\dot{\gamma}(a)$ and $\mathrm{tan}_{\dot\gamma(a)}$, the $g_{\dot\gamma(a)}$-orthogonal projection into $T_{p}L$. 
An instant $t_1$ is called
{\em $L$-focal} (and $\gamma(t_{1})$ {\em a focal point})
if there exists an $L$-Jacobi
field $J$ such that $J(t_1)=0$.
\end{definition}
One can check that the shape operator is symmetric since the connection is symmetric (recall \cite[Eq. (16)]{JavSoa15}).
Observe
that to compute $\mathrm{tan}_{\dot\gamma(a)} \nabla^{\dot\gamma(a)}_{u} \xi$, one formally needs a vector field  $\xi$ along $L$ which extends $\xi_{p}$,  but it turns out that this quantity does not depend on  the chosen extension, but only on $\xi_{p}$ (see \cite[Prop. 3.5]{JavSoa15}).

\subsection{Geodesic vector fields}
As we have stressed in \S \ref{section-finsler-metrics}, all the computations at a direction $v\in T_pM\setminus\{0\}$ can be made using a vector field $V$ without singularites defined on some neighborhood $\Omega$ of $p\in M$ with the affine connection $\nabla^V$. The choice of this $V$ is arbitrary, so, for example, it is possible to choose $V$ such that $\nabla^v_uV=0$ for all $u\in T_pM$ (see \cite[Prop. 2.13]{J20}). As, in this work, we will deal mainly with geodesics, it will be particularly convenient to choose a geodesic vector field $V$. In such a case, it is possible to relate some elements of the Chern connection with the Levi-Civita connection of the Riemannian metric $g_V$.
\begin{proposition}
\label{proposition-covariantderivative-finsler}
Let $V$ be a geodesic field on an open subset $U\subset M$ and $\hat{g}:=g_{V}$ denote
the Riemannian metric on $\Omega$  induced by the fundamental tensor $g$, 
 and let  $\widehat{\nabla}$ and $\widehat{R}$  be the Levi-Civita connection and   the Jacobi operator of
$\hat{g}$, respectively. Then, for any $X\in{\mathfrak X}(\Omega)$,
\begin{enumerate}[(i)]
\item $\widehat{\nabla}_XV=\nabla^V_XV$ and $\widehat{\nabla}_VX=\nabla^V_VX$,
\item $\widehat{R}_{V}X=R_{V} X$.
\end{enumerate}
As a consequence, the integral curves of $V$ are also geodesics of $\widehat{g}$, and 
the Finslerian Jacobi operator and Jacobi fields along the integral curves of $V$ coincide
with those of $\widehat{g}$.
\end{proposition}
A proof of this result can be found in \cite[Prop. 6.2.2]{Zhongmin-Shen} and also in \cite[Prop. 3.9]{J20}.

\subsection{Finsler submersions and foliations}
\label{section-Finslersubmersionsandfoliations}
One of the first  examples of Finsler foliations is the partition of $M$ into the fibers of a Finsler submersion.
Recall that a submersion  $\pi:(M,F)\rightarrow (B,\tilde{F})$  between Finsler manifolds is a  \emph{Finsler submersion} if 
$d\pi_p(B^F_p(0,1))=B^{\tilde F}_{\pi(p)}(0,1),$ for every $p\in M$,  where $B^F_p(0,1)$ and $B^{\tilde F}_{\pi(p)}(0,1)$ 
are the unit balls of the Minkowski spaces $(T_pM,F_p)$ and $(T_{\pi(p)}B,{\tilde F}_{\pi (p)})$ centered at $0$, respectively.
Examples and a few properties of  Finsler submersions can be found in \cite{Alexandrino-Alves-Javaloyes} and  \cite{Duran}.

Recall that we say that a geodesic $\gamma:I\subset \RR\rightarrow M$ is horizontal if $\dot\gamma(t)$ is an orthogonal vector to the fiber $\pi^{-1}(\pi(\gamma(t)))$ for every $t\in I$. Here we just  need to recall the lift property (see \cite[Theorem 3.1]{Duran}). 

\begin{proposition}
\label{proposition-duran-geodesics-submersion}
Let $\pi:(M,F)\to (B, \tilde{F})$ be a Finsler submersion. Then an immersed curve on $B$ is a geodesic if and only if
its horizontal lifts are geodesics on $M$. In particular, the geodesics of $(B,\tilde F)$ are precisely the projections of horizontal
geodesics of $(M,F)$.
\end{proposition}
This result allows us to get a geodesic field on $(M,F)$ that projects on a geodesic field on $(B,\tilde F)$, which will simplify many computations.
\begin{proposition}\label{redRiesub}
Let $\pi:(M,F)\to (B, \tilde{F})$ be a Finsler submersion and $V^*$ a geodesic field in some open subset $\tilde U$ of $B$. Then the horizontal lift $V$ of $V^*$ is  a geodesic vector field on $U=\pi^{-1}(\tilde U)$ and the restriction $\pi|_U:(U,g^F_V)\rightarrow (\tilde U,g^{\tilde F}_{V^*})$ is a Riemannian submersion, where $g^F$ and $g^{\tilde F}$ are the fundamental tensors of $F$ and $\tilde F$, respectively. 
\end{proposition}
\begin{proof}
That $V$ is a geodesic field follows from Proposition \ref{proposition-duran-geodesics-submersion}, while the last statement is a consequence of \cite[Prop. 2.2]{Duran}.
\end{proof}

A {\em Finsler foliation} of $M$ is a regular foliation $\F$ such that at every point $p\in M$, there exists a neighborhood $U$ in such a way that $\F|_U$ can be obtained as the fibers of a Finsler submersion. This is equivalent to the property of {\em transnormality}: 
\begin{align}\label{horgeo}
&\text{\em if a geodesic is orthogonal to one leaf of $\F$, then it is orthogonal to   all the }\\ &\text{\em fibers it meets. }\nonumber
\end{align}
Recall that a {\em singular foliation} of $M$  is a partition of $M$ by submanifolds, which are called leaves as in the case of regular foliations, with the following property: given $p\in M$, if $v\in T_pM$ is a vector tangent to a leaf, there exists a smooth vector field $X$ in a neighborhood of $p$ such that $X_p=v$ and $X$ is always tangent to the leaves of the foliation.
 This is  equivalent to saying that  for each $p$ there exists a neighborhood $U$ of $p$,   a neighborhood $B(0)$ of $0$ in $\RR^k$,  
and a submersion $\pi: U\to B(0)\subset\RR^{k}$ so that each
fiber $\pi^{-1}(x)$, with $x\in B(0)$, is contained in a leaf, and $P_p=\pi^{-1}(0)$
is a precompact open subset of the leaf $L_p$ that contains $p$ (see \cite[Lemma 3.8]{Alexandrino-Alves-Javaloyes}).
The submanifold $P_p$ is called \emph{plaque}. Note that $k$ here
is the codimension of $L_p$ and $B(0)$ can be considered to be a 
transverse submanifold $S_p$ to $P_p$. By definition, the leaves
$\F$ on $U$ must intersect $S_p$.    
Finally, we say that the singular foliation is \emph{Finsler} if  the property  \eqref{horgeo} holds. If $P_q$ is a plaque, 
 there exist future and past tubular neighborhoods denoted, respectively, by ${\mathcal O}(P_q,\varepsilon)$ and $\tilde{\mathcal O}(P_q,\varepsilon)$ of a certain radius $\varepsilon>0$ (see \cite{AlJav19}) 
and in these tubular neighborhoods a singular Finsler foliation can be characterized using the cylinders $\mathcal{C}^{+}_{r_{1}}(P_{q})=\{p\in M: d_F(P_q,p)=r_1\}$ and $\mathcal{C}^{-}_{r_{2}}(P_{q})=\{p\in M: d_F(p,P_q)=r_2\}$, where $d_F(P_q,p)$ is computed as the infimum of the lengths of curves from $P_q$ to $p$ and 
$d_F(p,P_q)$ as the infimum of the lengths of curves from $p$ to $P_q$. Recall that given $x\in {\mathcal O}(P_q,\varepsilon)$ (resp. $\tilde{\mathcal O}(P_q,\varepsilon)$), the plaque $P_x$  is the connected component which contains $x$ of the intersection of the leaf through $x$ with 
${\mathcal O}(P_q,\varepsilon)$  (resp. $\tilde{\mathcal O}(P_q,\varepsilon)$).
\begin{lemma}
\label{lemma-equidistant}
A  singular  foliation $\F$  is Finsler if and only if its leaves are locally equidistant, i.e., if its leaves satisfy  the following property:
if a point $x$ belongs to the future cylinder  $\mathcal{C}^{+}_{r_{1}}(P_{q})$ (resp. the past cylinder $\mathcal{C}^{-}_{r_{2}}(P_{q})$),  then the plaque
$P_{x}$ of the future (resp. past) tubular neighborhood is contained in  $ \mathcal{C}^{+}_{r_{1}}(P_{q})$ (resp. $\mathcal{C}^{-}_{r_{2}}(P_{q})$).
\end{lemma}
A proof of the above characterization can be found in \cite[Lemma 3.7]{Alexandrino-Alves-Javaloyes}.


\begin{figure}[hhhtb]
\label{figura1-cilindros-placas}
\includegraphics[scale=0.4]{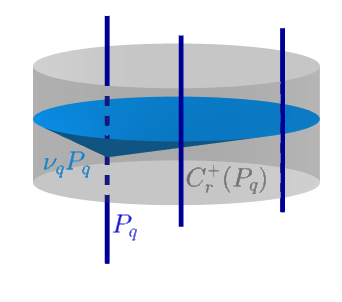}
\caption{\small{Figure generated by the software geogebra.org, illustrating a few concepts of \S \ref{section-Finslersubmersionsandfoliations},
such as a future cylinder $C_{r}^{+}(P_q)$ of a plaque $P_q$, other plaques contained
in this future cylinder and the normal cone at $q$ denoted by $\nu_{q}P_q$, i.e, the set of vectors $\xi\in T_{q}M$
so that $g_{\xi}(\xi,X)=0$ for all $X\in T_{q}P_q$. Here the leaves are the fibers of the Finsler submersion $\pi:(M,F) \to (B,\tilde{F})$,
where $M=\mathbb{R}^{3}$, $B=\mathbb{R}^{2}$, $\pi(x)=(x_1,x_2)$  and $F$ and $\tilde{F}$ are  
Randers metrics  having the Euclidean metrics and the winds $W=(\frac{1}{2},0,\frac{\sin^{2}(x_1)+1}{4})$ and
$\widetilde{W}=(\frac{1}{2},0)$ as Zermelo data, respectively. 
} }
\end{figure}

\begin{remark}
 Given a (singular) Riemannian foliation   $\F$ with closed leaves  
on a complete Riemannian manifold $(M,\mathsf{h})$ and an $\F$-basic vector field $W$, then
$\F$ is a (singular) Finsler foliation for the Randers metric with Zermelo data $(\mathsf{h},W)$,
c.f   \cite[Example 2.13]{Alexandrino-Alves-Javaloyes}.
 The converse is also true. In fact,   
according with  \cite[Theorem 1.1]{Alexandrino-Alves-Javaloyes}, 
  every singular Finsler foliation with closed leaves on  Randers spaces $(M,F)$ is produced in this way. 

\end{remark}



\section{Wilking's construction}
\label{section-wilkingsdistribution}


In \cite{WilkGAFA}, Wilking proved the smoothness of the
Sharafutdinov's retraction and studied dual foliations of singular Riemannian foliations (SRF for short). For that,
he used results on self-adjoint spaces, regular distributions along geodesics
and Morse-Sturm systems along these distributions, see also \cite{Gromoll-Walshap}.
These tools turn out to be quite useful in the study of the transversal geometry of SRF's (see e.g. \cite{LytchakThorbergsson1}).
Roughly speaking, given an SRF  $\F$ on $M$, 
for each horizontal geodesic $\gamma$  (that may cross singular leaves) 
one can define a distribution $t\to \WilkingH (t)$ 
along $\gamma$ so that $\WilkingH (t)$ is the normal distribution  $\nu_{\gamma(t)}(L)$ when $\gamma(t)$  lies in a regular leaf;  for the sake of simplicity we call this distribution \emph{Wilking's distribution}. 
In addition, when the leaves of $\F$ are closed
and $\pi: M\to M/\F$ is the canonical projection,  being the quotient $M/\F$ a manifold with the induced Riemannian metric,  one can identify the Jacobi fields along $\pi\circ\gamma$
with the solutions of a Morse-Sturm system along $t\to \WilkingH (t)$, the so-called \emph{transversal Jacobi field equation}. 

In this section we review the Finsler version of these objects; see Definition 
\ref{definition-wilkings-distribution-transverse-jacobi}.
We also present a few results on Finsler Jacobi fields,  whose proofs will be given  in \S \ref{apendice2}.  

\begin{definition}
Let $\gamma:I\subseteq \RR\rightarrow M$ be a geodesic of a Finsler manifold $(M,F)$ and $\selfadjointW$ a linear subspace of Jacobi fields so that $\gamma$ is orthogonal to the elements of
$\selfadjointW$, i.e., $g_{\dot{\gamma}}(\dot{\gamma},J)=0$ for every $J\in \selfadjointW$. The vector space $\selfadjointW$ is said to be \emph{self-adjoint} if $g_{\dot{\gamma}}(J_{1}',J_{2})=g_{\dot{\gamma}}(J_{1},J_{2}')$, 
for $J_{1}, J_{2} \in \selfadjointW$. 
\end{definition}

\begin{remark}
\label{remark-selfadjoint-at-zero}
One can prove that given two Jacobi fields $J_{1}$ and $J_{2}$ along $\gamma$ then  
$g_{\dot{\gamma}}(J_{1}',J_{2}) -g_{\dot{\gamma}}(J_{1},J_{2}')$ is constant on the interval $I$. 
Therefore to show that a vector space $\selfadjointW$  of Jacobi fields  is self-adjoint, it suffices to check  that 
$g_{\dot{\gamma}(t_0)}(J_{1}'(t_0),J_{2}(t_0))=g_{\dot{\gamma}(t_0)}(J_{1}(t_0),J_{2}'(t_0))$ for some $t_0\in I$  (see for example \cite[Prop. 3.18]{JavSoa15}). 
\end{remark}

In the next lemma, we present the distribution $t \to \WilkingH(t)$ associated with a general subspace $\selfadjointV$
of  $\selfadjointW$ as well as the (generalized) transversal Jacobi field equation.

\begin{lemma}
\label{lemma-dimVconstant}
Let $\selfadjointW$ be an $(n-1)$-self-adjoint vector space of Jacobi fields orthogonal to a geodesic $\gamma:I\subseteq\RR\rightarrow M$ on a Finsler manifold $(M^{n},F).$ 
Let $\selfadjointV$ be a linear subspace of $\selfadjointW$ and hence also a self-adjoint space. 
 Define the subspace $\WilkingV(t)$ of $T_{\gamma(t)}M$ for every $t\in I$ as
$$\WilkingV(t):=\{J(t) | J \in \selfadjointV \} \oplus \{ J'(t) | J\in \selfadjointV, J(t)=0\}.$$
\begin{enumerate}
\item[(a)]
Then $\dim \WilkingV(t)=\dim \selfadjointV$ for every $t\in I$. Furthermore, the second summand is trivial for almost every $t$. 
\item[(b)] Let 
$t \to \WilkingH(t)$  be the orthogonal complement of $\WilkingV(t)$  with respect to $g_{\dot{\gamma}(t)}$, i.e., $w\in \WilkingH(t)$ if and only if   
$g_{\dot{\gamma}(t)}(u,w)=0$ for all $u\in \WilkingV(t)$.
 Let $(\cdot)^{\h}$ and $(\cdot)^{\vp}$ be the
orthogonal projections (with respect to $g_{\dot{\gamma}}$) into $\WilkingH(t)$ and $\WilkingV(t)$, respectively. Then  if $J\in \selfadjointW$, 
$J^{\h}$ fulfills the transversal Jacobi equation, i.e., 
\begin{equation} 
\label{eqJacWil2} 
(\Dh)^2 (J^{\h}) + (R_{\dot{\gamma}}(J^{\h}))^\h - 3 (\mathbb{A}_\gamma)^2( J^{\h}) = 0,
\end{equation} 
where $\Dh$ is the induced connection on the horizontal bundle,   which is defined as  $\Dh(X)=((X^\h)')^\h$, 
and  $\mathbb{A}_{\gamma}$ is the   
 O' Neill tensor along the geodesic $\gamma$, i.e., 
$\mathbb{A}_\gamma (X) := \big( (X^\h)'\big)^\vp + \big(  (X^\vp)'\big)^\h$
for every $X\in \mathfrak{X}(\gamma)$. 
\item[(c)] If $J_1,\ldots,J_r$ is a basis of $\selfadjointV$ with $r=\dim \selfadjointV$ such that at $t_0\in I$, $J_k(t_0)=\ldots=J_r(t_0)=0$ and $J_i(t_0)\not=0$ for $i=1,\ldots,k-1$, then 
\[J_1(t),\ldots,J_{k-1}(t),\frac{1}{t-t_0}J_{k}(t),\ldots, \frac{1}{t-t_0}J_{r}(t)\]
 is a continuous basis of $\WilkingV(t)$ for every $t$ in  $(t_0-\varepsilon,t_0+\varepsilon)\cap I\setminus \{t_0\}$  for some $\varepsilon>0$, with $\lim_{t\rightarrow t_0} \frac{1}{t-t_0}J_{i}(t)=J_i'(t_0)$ for all $i=k,\ldots,r$.     
\end{enumerate}
\end{lemma}
\begin{proof}
  First, observe that given $t_0\in I$, there exists a neighborhood $U\subset M$ of $\gamma(t_0)$ which admits a geodesic vector field $V$ such that $\gamma(t)\subset U$ and $V_{\gamma(t)}=\dot\gamma(t)$ for all $t\in I_\varepsilon=(t_0-\varepsilon,t_0+\varepsilon)\cap I$ for some $\varepsilon>0$.
	From Proposition \ref{proposition-covariantderivative-finsler}, we conclude that $\selfadjointW$ and $\selfadjointV$ are also self-adjoint spaces of (Riemannian) Jacobi fields
with respect to the Riemannian metric $\hat{g}:=g_{V}$. Since the lemma is already true for self-adjoint spaces of Riemannian Jacobi
fields (see  \cite[Chapter 1]{Gromoll-Walshap}), it is also valid in the Finsler case for $t\in I.$ 
 Indeed, part (a) follows from \cite[Lemma 1.7.1]{Gromoll-Walshap} (for the triviality of the second summand, observe that the interval $I$ can be covered by a countable number of intervals $I_{\varepsilon_i}$ with geodesic vector fields $V_i$, $i\in \N$, as above).  Part (b) follows from the equation above \cite[Eq. (1.7.7)]{Gromoll-Walshap} observing that   
  if  $A(t):\WilkingV(t)\to \WilkingH(t)$ and its dual $A^{*}(t):\WilkingH(t)\to \WilkingV(t)$ are the linear operators
defined below \cite[Eq. (1.7.4)]{Gromoll-Walshap} (which are also naturally defined for Finsler metrics), then $\mathbb{A}_{\gamma}|_{\WilkingH}=-A^{*}$ and 
$\mathbb{A}_\gamma|_{\WilkingV}=A$. 
In fact, when $t_0$ is regular (the second summand in part $(a)$ is trivial),  for every $u\in\WilkingV(t)$, there exists a Jacobi field $J\in \selfadjointV$ such that  $J(t_0)=u$. Then $A(t_0) u=(J ')^\h(t_0)$, but as $J^\vp=J$, it follows that $\mathbb{A}_\gamma|_{\WilkingV(t_0)}=A(t_0)$. Moreover, as 
\begin{multline*}
g_{\dot\gamma}(\mathbb{A}_\gamma|_{\WilkingH(t_0)} (X), Z) =  -g_{\dot\gamma}( X, \mathbb{A}_\gamma|_{\WilkingV(t_0)}(Z)) 
  =g_{\dot\gamma}( X, -A(t_0)(Z))\\= g_{\dot\gamma}(-A^{*}(t_0) (X), Z), 
\end{multline*}
which allows us to conclude that $\mathbb{A}_\gamma|_{\WilkingH(t_0)}=-A^{*}(t_0)$. This implies, taking into account the equation  above \cite[Eq. (1.7.7)]{Gromoll-Walshap}, that Eq. \eqref{eqJacWil2}  holds for every regular instant, and, by continuity, for all $t\in I$, as regular instants are dense by part $(a)$.   We stress that, by Proposition \ref{proposition-covariantderivative-finsler}, the O'Neill tensor of $\gamma$ computed with $g_V$ coincides with that computed with $F$.  Finally, for  part (c), see the proof of \cite[Lemma 1.7.1]{Gromoll-Walshap}. 
\end{proof}

As we will see below, an example of $(n-1)$-self-adjoint space $\selfadjointW$ 
is the space of $L$-Jacobi fields for a submanifold $L$  along a geodesic $\gamma$ which is orthogonal to $L$. 
In the case where $L$ is a fiber of a Finsler submersion
$\pi:M\to B$, $\selfadjointV$ is  the space of  holonomic Jacobi fields, i.e.,
those  Jacobi fields whose $\pi$-projections are zero. Also in this case,  the transversal Jacobi field equation along $\gamma$ 
will be identified with a Jacobi field equation
along a geodesic in $B$ (see  Remark \ref{remark-transverseJacobifield}). 

\begin{lemma}
\label{lemma-W-selfadjoint}
Let $L$ be a submanifold on a  Finsler manifold $(M^{n},F)$ and 
$\gamma:I\subset\RR\rightarrow M$ a geodesic orthogonal to $L$ at $t_0\in I.$ 
Consider $\selfadjointW$ the vector space of $L$-Jacobi fields orthogonal to $\gamma$. 
Then $\selfadjointW$ is  a self-adjoint space   of dimension $n-1$. 
\end{lemma}
\begin{proof} See  \S \ref{lema34}.
\end{proof}

\begin{proposition}
\label{proposition-jacobifield-variation}
Let $(M,F)$ be a Finsler manifold and $L$ a submanifold of $M$. Given a geodesic $\gamma:I=[a,b]\subset\RR\rightarrow M$ orthogonal to $L$ at the instant $t_0$, we have that a vector field $J$ along $\gamma$ is $L$-Jacobi if and only if it is the variation vector field of a  variation whose longitudinal curves are $L$-orthogonal geodesics. 
\end{proposition}
\begin{proof} See  \S \ref{prop35}.
 \end{proof}

When $L$ is a fiber of a Finsler submersion and $J$ is an $L$-Jacobi field tangent to the fibers, then $J$ turns out to be
the variation vector field of a variation determined by  end-point maps $\eta^r_{\xi}$, as we will see in the next lemma. 

\begin{lemma}
\label{lemma-variacao-jacobiholonomia}
Consider a Finsler submersion $ \pi:(M,F)\to (B,\tilde F)$,  and set $L=\pi^{-1}(b)$ for some $b\in B$. 
Let  $\gamma:I\rightarrow M$ be a  geodesic orthogonal to $L$ at $t_0\in I$ and $t\to J(t)$ be an $L$-Jacobi field along $\gamma$ with $J(t_{0})$ tangent
to $T_{p}L$ where $p=\gamma(t_{0})$. Also assume that $J$ is $g_{\dot\gamma}$-orthogonal to $\gamma$. Then the following items are equivalent:
\begin{enumerate}
\item[(a)] $J$ is a vertical (holonomy) Jacobi field, i.e., $J$ is always tangent to the fibers of $\pi$. In other words $J^{\h}=0$.
\item[(b)]  There exists a curve $\beta:(-\varepsilon,\varepsilon)\subset \RR\rightarrow L$ 
with $J(t_0)=\dot{\beta}(0)$, such that if $\xi$ is the normal basic vector field along $L$ with $\xi_p=\dot{\gamma}(t_{0})$ and $\psi:I\times(-\varepsilon,\varepsilon)\rightarrow M$ is the variation 
$(t,s)\to \psi(t,s)=\gamma_{\xi_{\beta(s)}}(t-t_0)=\eta^{t-t_{0}}_{\xi}(\beta(s))$,  then
$J(t)=\frac{\partial}{\partial s} \psi(t,0)$.
\end{enumerate}
\end{lemma}
\begin{proof}  See  \S \ref{lema36}.
\end{proof}

\begin{definition} 
\label{definition-wilkings-distribution-transverse-jacobi}
Let $\mathcal F$ be a partition by stratified submanifolds of $M$ such that there exists
 an open neighborhood $U\subset M$ where the partition $\mathcal F|_{U}$ is a regular Finsler foliation of $U$.  Given a geodesic $\gamma:I\subseteq \RR\rightarrow M$ and $t_0\in I$ such that $\gamma(t_0)\in U$ and  $\dot\gamma(t_0)\in \nu(L_{\gamma(t_0)})$, denoting by $L_p$  the fiber of $\mathcal F|_{U}$ which contains $p\in U$,
let  $\selfadjointW$ be the self-adjoint space of $L_{\gamma(t_0)}$-Jacobi fields orthogonal to $\gamma$ defined in  Lemma \ref{lemma-W-selfadjoint}
and  $\selfadjointV$ the subspace of  $\selfadjointW$ such that $J \in \selfadjointV$ iff $J(t)\in T_{\gamma(t)}L_{\gamma(t)}$
 for all $t\in (t_{0}-\varepsilon, t_{0}+\varepsilon)$, where $\varepsilon>0$ is small enough  to guarantee that 
$\gamma(t_0-\varepsilon,t_0+\varepsilon)\subset U$.    
Then the distributions $t\to \WilkingV (t)$, $t\to \WilkingH (t)$ and  the differential equation  \eqref{eqJacWil2}  
associated with the self-adjoint spaces $\selfadjointW$ and $\selfadjointV$ 
are called  \emph{the Wilking's distributions and the transversal Jacobi field equation along $\gamma$  starting at} $ \{\gamma(t_0),L_{\gamma(t_{0})}\}$. 
\end{definition}
 

\begin{figure}[hhtb]
\includegraphics[scale=0.4]{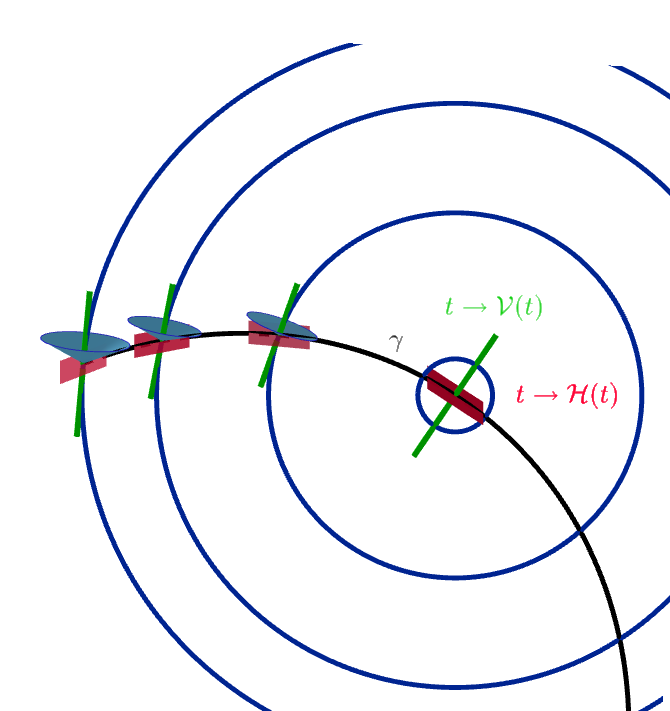}
\caption{\small{ Figure generated by the software geogebra.org, illustrating the Wilking's distributions
$t\to \mathcal{H}(t)$ and $t\to \mathcal{V}(t)$ associated with the singular Finsler
foliation $\F_{\pi}=\{\pi^{-1}(c)\}$ on the neighborhood   $B_{\frac{3}{2}}(0)\subset \mathbb{R}^{3}$ where $\pi:\mathbb{R}^{3}\to\mathbb{R}^{2}$ is defined as
$\pi(x)=(x_1^{2}+x_{2}^{2},x_3).$ The Finsler metric  $F$ restricted to  $B_{3/2}(0)$ 
is the Randers metric  having the Euclidean metric and  the wind $W=(-\frac{x_2}{2},\frac{x_1}{2},0)$ as Zermelo data. 
}} 
\end{figure}

 

    


\begin{remark}
Note that even when $\F$ is a singular foliation, the distributions 
$t\to \WilkingV (t)$, $t\to \WilkingH (t)$ and the differential equation \eqref{eqJacWil2} 
 are well defined for all $t$ including the case where $\gamma(t)$ is a singular point. 
Also note that the transversal Jacobi field equation \eqref{eqJacWil2} 
is a Morse-Sturm system; see \cite{GMPiccioneTausk,Lytchak-Jacobifield,LytchakThorbergsson1,PaoloDaniel}. 
 Although  the distribution $\WilkingV (t)$ coincides with the tangent space
of $L_{\gamma(t)}$  until $\gamma$  crosses a singular leaf,  it does not follow straightforwardly that $\WilkingV (t)$ should  always be tangent to the  regular leaves for all $t$ (even at regular leaves after crossing the first singular leaf).  
This is the case  when $\F$ is the foliation
studied in \S \ref{section-singular-analytic-submersion}  obtained as the fibers of an analytic Finsler submersion. 
\end{remark}


\begin{remark}
\label{remark-transverseJacobifield}
By applying Proposition \ref{proposition-jacobifield-variation} we can infer  useful  natural interpretations
of the transversal Jacobi field equation.
Let $\pi: M\to B$ be a Finsler submersion and $L$ a fiber, i.e., $L=\pi^{-1}(b)$, for any $b\in B$. Let $p\in L$ and $\gamma_p$ be a horizontal
geodesic starting at $p$, i.e., $p=\gamma_{p}(0)$. Then the  
transversal Jacobi  equation along $\gamma_{p}$  starting at $L_{\gamma_p(0)}$ can be identified with the Jacobi field equation
along the geodesic $\pi\circ \gamma_{p}$ in $B$.    
Indeed, the transversal Jacobi equation is the horizontal lift to $M$ of the Jacobi equation on $B$. To check this, consider a geodesic vector field $V$ whose integral curves are horizontal geodesics of $(M,F)$, being $V^*=d\pi\circ V$, and recall 
that, by Proposition \ref{redRiesub}, $\pi:(M,g_V)\rightarrow (B,\tilde g_{V^*})$ is a Riemannian submersion. Then use the relation between the Levi-Civita connection of $(B,\tilde g_{V^*})$ and the horizontal part of the Levi-Civita connection of $(M,g_V)$ when applied to basic vector fields (see \cite[Lemma 1 (3)]{Oneill66}) and the relation between the horizontal part of their curvature tensors in \cite[Theorem 2 \{4\}]{Oneill66} taking into account Proposition \ref{proposition-covariantderivative-finsler}.
In addition, if $x\in L$ and $\gamma_x$ is a horizontal geodesic starting at $x$
so that $\pi\circ \gamma_{x}=\pi\circ \gamma_{p}$, then the transversal Jacobi field equation along $\gamma_{p}$  starting at $L_{p}$
and the transversal Jacobi field equation along $\gamma_{x}$  starting at $L_{x}$ can be identified with each other.
As we will see from Lemmas \ref{SubAnal-lemma1} and \ref{SubAnal-lemma2}, this interpretation will also hold for 
analytic singular Finsler submersions. 

\end{remark}

\section{Proof of Theorem \ref{SubAnal}}
\label{section-singular-analytic-submersion}

In this section we prove Theorem \ref{SubAnal} adapting an argument of \cite{Alexandrino-IRSS} and using Wilking's distributions, recall
 Definition \ref{definition-wilkings-distribution-transverse-jacobi}.  Let $\pi
:M^n\to B^k$ be an analytic map. In the first subsection, assuming $\pi$ is a Finsler submersion on an open subset of $M$, we  prove that $\pi$ is a Finsler submersion on  the regular part  $M_0$ and the regular fibers are equifocal submanifolds. In the second part, we will assume that the fibers constitute a singular  foliation and will prove the transnormality.

\subsection{Equifocality}

 Let $q\in M$ be a regular point of $\pi$, $c=\pi(q)$ and  $L=\pi^{-1}(c)$. We will consider a unit basic vector field  $\xi$ along a neighborhood of $q$ in  $L$,  i.e., 
$\xi$ is orthogonal to $L$,  projectible and $F(\xi)=1$. 
 We will show that $L$ is a regular fiber of $\pi$ by extending this basic vector field along $L$ (Lemma \ref{firstapproach})  and will check  part $(b)$ of Theorem \ref{SubAnal}. 
More precisely,   in Lemma \ref{SubAnal-lemma1}, we will prove that $\eta^r_\xi(L)$ is contained in a fiber of $\pi$ and in
 Lemmas \ref{SubAnal-lemma3} and  \ref{lemma-equifocality}, that $d\eta^r_\xi$ has constant rank for all $r$.

\begin{lemma}\label{firstapproach}
	If $p\in M$ is a regular point of $\pi$ and there exists a neighborhood $U$ of $p$ in (the regular part of)  $L_p$ such that $F$ is constant on basic vector fields on $U$, then $L_p$ is regular and $F$ is constant on basic vector fields on $L_p$.
\end{lemma}
\begin{proof}
Let us denote by $L^R_p$ the connected component of $p$ of this regular part (which is open)  and $\xi$ a basic vector field on $L^R_p$, which is analytic as all the data is analytic. 
  Then as, by hypothesis,  the function $F\circ \xi:L^R_p\rightarrow \RR$ is constant in a neighborhood of $p$, 
	it must be constant everywhere. Now consider a point $q$ in the boundary of $L^R_p$. As $F\circ \xi$ is constant,  if one considers a sequence of points $\{q_n\}\subset L_p^R$ such that $\lim_{n\to \infty}q_n=q$, then the sequence of vectors $\xi_{q_n}$ is precompact in $TM$ and it converges to some $v_q\in T_qM$ up to a subsequence.  Let $\tilde \xi=d\pi(\xi)$, which is well-defined because $\xi$ is basic.   By continuity, the vector $v_q$  projects to $\tilde\xi$, and since this can be done for every $\tilde\xi\in T_{\pi(p)}B$  (with the corresponding basic lift),  it turns out that the points of the boundary are regular and   $F$ is constant on $\xi$. 
\end{proof}
\begin{remark}\label{saturated}
 As a consequence of Lemma \ref{firstapproach},  the assumption that $\F_\pi$ is a Finsler foliation restricted to some open subset $U\subset M$ implies that there exists a saturated open subset where $\pi$ is a Finsler submersion. Namely, consider $\pi(U)$, which is open because $\pi$ is a submersion and then $\pi|_{\pi^{-1}(\pi(U))}$ is a Finsler submersion for a certain Finsler metric on $\pi(U)$. 
\end{remark}

	\begin{lemma}
	\label{SubAnal-lemma1}
	For all $r\in\RR$, there exists a value $d\in B$ such that $\eta^r_{ \xi}(L)\subset \pi^{-1}(d)$. 
	\end{lemma}
	\begin{proof}
 Assume that $r>0$ (the case $r<0$ is analogous and $r=0$ is trivial).  Consider $x,y\in L$ and set $\tilde I:=\{ t_{0}\in [0,r]\}$ such that $\pi(\gamma_{\xi_x}(t))= \pi(\gamma_{\xi_y}(t)),$ for $t\leq t_{0} \}$.
	Note that the set $\tilde I$ is not empty, because $\pi$  restricted to a neighborhood of  $L$ is a Finsler submersion  (recall Remark \ref{saturated}).  The set $\tilde I$ is closed due to continuity and it is an open set because $\pi$ is an analytic map and geodesics of an analytic Finsler metric are also analytic. Therefore $\tilde I=[0,r]$ and this concludes the proof.
	\end{proof}

Fixed $x\in L$,  and consider the geodesic $\gamma_{\xi_x}$. Since,   by Remark \ref{saturated}, $\pi$ is a Finsler submersion   on a neighborhood of $x$,  we can define along the 
geodesic $\gamma_{\xi_x}$ the Wilking's distribution pair  $t\to (\WilkingH(t),\WilkingV(t))$  (recall Definition \ref{definition-wilkings-distribution-transverse-jacobi}).

\begin{lemma}\label{lemma:V=L}
 If $x\in L$ and  $p=\gamma_{\xi_x}(t_0)$ is a regular point of $\pi$, then the Wilking's distribution $\mathcal V(t_0)$ coincides with the tangent space to $L_{p}$. Moreover, the distributions $t\rightarrow \mathcal V(t)$ and $t\rightarrow \mathcal H(t)$
are analytic and the singular points of $\gamma_{\xi_x}$ (lying in a singular level set) are isolated.
\end{lemma}
\begin{proof}
Recall by part (c) of Lemma \ref{lemma-dimVconstant} that there is a basis $J_1,\ldots,J_{r}$ of  
$\selfadjointV$ with $r=\dim \selfadjointV$ such that at $t_0$, $J_k(t_0)=\ldots=J_r(t_0)=0$ and $J_i(t_0)\not=0$ for $i=1,\ldots,k-1$. Then 
\begin{equation}\label{Js}
J_1(t),\ldots,J_{k-1}(t),\frac{1}{t-t_0}J_{k}(t),\ldots, \frac{1}{t-t_0}J_{r}(t)
\end{equation}
 is a continuous basis of $\WilkingV(t)$ for every $t$ in  $(t_0-\varepsilon,t_0+\varepsilon)$  for some $\varepsilon>0$, with 
$\lim_{t\rightarrow t_0} \frac{1}{t-t_0}J_{i}(t)=J_i'(t_0)$ for all $i=k,\ldots,r$. By Lemma \ref{lemma-variacao-jacobiholonomia},  
$J_i(t)=d\eta^t_{\xi}(v_i)$ for some vector $v_i\in  T_{x}L $ for all $i=1,\ldots,r$, and  by Lemma \ref{SubAnal-lemma1}, all $J_i(t)$ are tangent to $L_{\gamma_{\xi_x}(t)}$ for all $t$.  By continuity,  $J'_j(t)$ is also tangent to $L_{\gamma_{\xi_x}(t)}$  for $j=k,\ldots,r$. 
This implies that $\WilkingV(t_0)$ is contained in 
 the tangent space to $L_{p}$. As they have the same dimension, they coincide. The analyticity of $\WilkingV$ and $\WilkingH$ 
follows from the basis of $\WilkingV(t)$ in \eqref{Js}, as the vector fields $\tilde{J}_j(t)=\frac{1}{t-t_0}J_{j}(t)$ extend analytically to $t_0$ making $\tilde{J}_j(t_0)=J_j'(t_0)$ as $J_j(t_0)=0$ and $J_j$ is analytic, for $j=k,\ldots,r$.  
 Up to the singular points of $\gamma_{\xi_x}$,  consider an analytic frame of the horizontal space $\WilkingH$, namely, analytic vector fields along $\gamma_{\xi_x}$, $H_1,\ldots,H_{n-r}$ such that 
 $H_1(t),\ldots,H_{n-r}(t)$ is a basis of $\WilkingH(t)$ for every $t$ (this can be obtained using the $D^\h$-parallel transport). Then when chosen an analytic frame along 
 $\pi\circ \gamma_{\xi_x}$, $E_1,\ldots,E_{n-r}$, the determinant of the transformation matrix  with respect to $d\pi\circ H_1,\ldots,d\pi\circ H_{n-r}$  can only have isolated zeroes or being zero everywhere. As it is not zero close to $t_0$, it has only isolated zeroes. Even if one cannot ensure the existence of a global frame for $\pi\circ \gamma_{\xi_x}$, it is enough to recover its domain with frames which intersect in open subsets to conclude that the singular points along $\gamma_{\xi_x}$ are isolated.  Observe that, by 
 Lemma~\ref{firstapproach}, the singular points of $\pi$ can only lie in singular levels. 

\end{proof}

\begin{lemma}
	\label{SubAnal-lemma2}
	If $p=\gamma_{\xi_x}(t)$ is a regular point of $\pi$, then $\gamma_{\xi_x}$ is orthogonal to $L_{p}$.   
\end{lemma}
\begin{proof}
	
	Observe that $f_i(t)=g_{\dot\gamma_{\xi_x}}(\dot\gamma_{\xi_x}(t),J_i(t))$  and $h_i(t)=g_{\dot\gamma_{\xi_x}}(\dot\gamma_{\xi_x}(t),J'_i(t))$ are   analytic  functions   for each $J_i\in \selfadjointV$  in the basis of part (c) of Lemma \ref{lemma-dimVconstant},  and if  $\gamma_{\xi_x}(t)$ is a regular point of $\pi$,  then,  by Lemma \ref{lemma:V=L},  $T_{\gamma_{\xi_x}(t)}$  $L_{\gamma_{\xi_x}(t)}=\mathcal V(t)$. 
	As $\pi$ is a Finsler submersion in a neighborhood of $L$ (recall Remark \ref{saturated}), $f_i$ and $h_i$ are identically zero, and, in particular,  by part (c) of Lemma \ref{lemma-dimVconstant},  $\gamma_{\xi_x}$ is orthogonal to $L_p$.

\end{proof}

\begin{lemma}
	\label{pontoregular}
	 If there exists a neighborhood $U$ of $M$ such that  $\F_\pi|_U$ 
	is a Finsler foliation, 
	then $\F_\pi$ is a Finsler foliation when restricted to the open subset of regular points of $\pi$. 
\end{lemma}
\begin{proof}
  Fix $q \in U$.	The first observation is that, by Remark \ref{saturated}, $\pi$ is a Finsler submersion when restricted to $\pi^{-1}(\pi(U))$ and 
	$L_{q}$ is a closed submanifold. Given $p \in M$, a regular point of $M$, there exists a minimizing geodesic $\gamma:[0,b]\rightarrow M$ from 
	$L_{q}$ to $p $, and it follows from the Finsler Morse index Theorem (see \cite{Peter06}) that there is no focal point on $[0,b)$. As $\gamma$ minimizes the distance from $L_{q}$, it must be orthogonal to $L_{q}$, and then, by 
	Lemma \ref{SubAnal-lemma2}, orthogonal to all the regular fibers of $\pi$. Now  consider the basic vector field $\xi$ along $L_{q }$ such 
	that $\xi_{\gamma(0)}=\dot\gamma(0)$. Then the map $\eta^t_\xi:L_{q }\rightarrow L_{\gamma(t)}$ is a  diffeomorphism on a neighborhood of 
	$\gamma(0)$ for every $t\in [0,b)$. Fix an instant $t_0\in (0,b)$ such that $\gamma(t_0)$ is a regular point of $\pi$ (recall that, by 
	Lemma \ref{lemma:V=L}, singular points are isolated along $\gamma$), then by continuity and the compactness of $[0,t_0]$, 
	it is possible to choose a neighborhood  $\widetilde{U}\subset L_{q }$ 
		where $\eta^t_\xi:\widetilde{U}\rightarrow L_{\gamma(t)}$  is a diffeomorphism onto the image for all $t\in [0,t_0]$.  
		Given a vector $v$ horizontal to $\widetilde{U}_0=\eta^{t_0}_\xi(\widetilde{U})$ at $p_{0} =\eta^{t_0}_\xi(q_0)$,  for some $q_0\in \tilde U$, 
		consider the Wilking's  distributions $\WilkingH^{q_0}(t_0)$ and 
		$\WilkingV^{q_0}(t_0)$ along the horizontal curve $\gamma_{\xi_{q_0}}$, 
		and let $v^\h$ be the projection of $v$ to $\WilkingH^{q_0}(t_0)$. 
		Consider the $D^\h$-parallel vector field $X$ along $\gamma_{\xi_{q_0}}$ 
		such that $X(t_0)=v^\h$  and  the vector field $Y$ along  $\gamma_{\xi_{q_0}}$ 
		 orthogonal to $\eta^t_{\xi}(\widetilde{U})$ at each $t\in [0,t_0]$ 
		such that its Wilking's projection to $\WilkingH^{q_0}$ is $X$  (recall \cite[Lemma 2.9 (a)]{Alexandrino-Alves-Javaloyes}). 
		Let $v_1$ be the orthogonal vector   to $L_{q }$ at $\gamma(0)$ such that $d\pi(v_1)=d\pi(Y(0))$.  
		Now repeat the process to obtain a vector field  along $\WilkingH^{\gamma(0)}$ (the horizontal Wilking's distribution along $\gamma$)  
		and $\tilde Y$ horizontal to $\eta^t_{\xi}(\widetilde{U})$ such that its Wilking's projection to $\WilkingH^{\gamma(0)}(t)$, 
		$\tilde X=\tilde Y^\h$,  
		is $D^\h$-parallel and $\tilde X(0)=v_1^\h$. Let us observe that if we prove that 
\begin{equation}\label{projection}
d\pi(Y(t))=d\pi(\tilde Y(t))\quad \text{for all $t\in [0,t_0]$},
\end{equation}
 then, for $t$ close to $0$, $F(Y(t))=F(\tilde Y(t))$, because $\pi$ is a Finsler submersion, but, by analyticity, this will be true for all $t\in [0,t_0]$. This implies that $F(v)=F(\tilde Y(t_0))$, and in particular, that $F$ is constant on basic vector fields 
along $\widetilde{U}_0$. By Lemma \ref{firstapproach}, $F$ is constant on basic vector fields along the whole fiber $L_{\gamma(t_0)}$. Finally, by continuity, this is also true for $L_{p }$.
So, let us prove Eq. \eqref{projection}, which is equivalent to prove that   $d\pi(X(t))=d\pi(\tilde X(t)).$ 
 This  equivalence  follows from the fact that  $d\pi(\WilkingV^{q_0}(t))=d\pi(\WilkingV^{\gamma(0)}(t))=\{0\}$, because 
$T_{\gamma_{\xi_{q_0}}(t)} \eta^t_\xi(\widetilde{U})=\WilkingV^{q_0}(t)$,
$T_{\gamma(t)} \eta^t_\xi(\widetilde{U})=\WilkingV^{\gamma(0)}(t)$, and 
$\eta^t_\xi (\widetilde{U})\subset L_{\gamma_{\xi_{q_0 }}(t)}=L_{\gamma(t)}$  
(recall that we assume that $\eta^t_\xi|_{\tilde U}$ is a diffeomorphism onto the image for $t\in [0,t_0]$).  
By analyticity, we only have to prove that this holds for $t$ close to $0$. But this is true, because 
the $D^\h$-parallel transport in a Riemannian submersion is the horizontal lift  of the parallel transport in the base. 
In our case, the $D^\h$-parallel transport coincides with the lift of the parallel transport induced by $\tilde g_{V^*}$, where $V^*$ is a (local) geodesic vector field in the base tangent to $\pi\circ\gamma$, considering the Riemannian metric $g_{V}$ on $M$, where $V$ is the horizontal lift  of $V^*$, 
being $g$ and $\tilde g$, the fundamental tensors of the Finsler metrics on $M$ and $B$, respectively (recall Proposition \ref{redRiesub}). So, the proof is concluded.   

\end{proof}

\begin{remark}
Observe that as a consequence of Lemmas \ref{lemma:V=L}, \ref{SubAnal-lemma2} and \ref{pontoregular}, given a geodesic which is horizontal in one regular point, then it is horizontal in all the regular points and, by continuity (using Lemma \ref{lemma:V=L}) the Wilking distribution does not depend on the regular instant chosen as initial point. Let us see that the solutions to the transversal Jacobi equation of two horizontal geodesics with the same projection can be identified. 
\end{remark}

\begin{lemma}\label{comparingJh}
Consider a geodesic $\gamma$  of $(M,F)$ which is horizontal at the regular points. It holds that
\begin{enumerate}[(a)]
\item  if $X$ is a solution of the transversal Jacobi equation of $\gamma$ in Eq. \eqref{eqJacWil2}, then along the regular instants of $\gamma$, $d\pi\circ X$ is a Jacobi field of $\pi\circ \gamma$, and the solutions of the transversal Jacobi equation are characterized by this property,
\item  if  $\alpha$ is another geodesic which is horizontal in the regular points with  $\pi\circ \gamma=\pi\circ\alpha$, then there exists a vector field $Y$ along $\alpha$ which fulfills the transversal Jacobi equation of $\alpha$ and such that  $d\pi\circ X=d\pi\circ Y$ 
 and $g_{\dot\gamma}(X,X)=g_{\dot\alpha}(Y,Y)$, 
\item $\gamma$ and $\alpha$ provide the same conjugate points of the transversal Jacobi equation.
\end{enumerate}
\end{lemma}
\begin{proof}
For part $(a)$, recall Remark \ref{remark-transverseJacobifield}, and use that the singular points are isolated (see Lemma \ref{lemma:V=L}). For part $(b)$, observe that part $(a)$ implies part $(b)$ when we restrict $\gamma$ and $\alpha$ to an interval with regular points. Analyticity implies that $Y$ can be extended to the whole interval with the required properties. Part $(c)$ is a straightforward consequence of part $(b)$.
\end{proof}

\begin{lemma}
\label{SubAnal-lemma3}
If $d$ is a regular value, then $\eta^r_{\xi}: L\rightarrow \pi^{-1}(d)$  is a diffeomorphism. 
\end{lemma}
\begin{proof}
By Lemma \ref{SubAnal-lemma2},   $\dot{\gamma}_{\xi_p}(r)$ is orthogonal to $\pi^{-1}(d)$. As $\pi$ is a Finsler submersion on $\pi^{-1}(d)$ (recall Lemma \ref{pontoregular}), we can extend the normal vector $\dot{\gamma}_{\xi_p}(r)$ to a unit basic vector field $\tilde{\xi}$ along the fiber $\pi^{-1}(d)$. 
It is also possible to check that  $\eta^{-r}_{ \tilde{\xi}}$  is the inverse of 
$\eta^r_{\xi}$.
\end{proof}

 We have just proved in Lemma  \ref{SubAnal-lemma3}  
 that $\eta^r_{\xi}:\pi^{-1}(c) \to \pi^{-1}(d)$ is a diffeomorphism when $d$ is a regular value, which implies that $d\eta^r_\xi$ has constant rank. We have to check now that
 $d\eta^r_\xi$ has constant rank  when $d$ is a singular value. 
This will be done in Lemma  \ref{lemma-equifocality}.

\begin{lemma}
\label{lemma-equifocality}
Consider $\eta^r_{\xi}:L\rightarrow \pi^{-1}(d)$ with $d\in B$ a singular value. Then  $\dim \rank d\eta^r_{\xi}$ is constant 
along $L$. 
\end{lemma}
\begin{proof}

 As $L$ is connected, it is enough to prove that $\dim \rank d(\eta^r_{\xi})_{x}$ is locally constant. 
Consider a  point $p\in L$. Since the transversal Jacobi field equation along $\gamma_{\xi_p}$ starting at  $L$  is a Morse-Sturm system, 
there exists a $\delta>0$ so that any two points in $I_\delta=(r-\delta, r+\delta)$ are not conjugate to each other. In other
words, if $X$ is a solution of the transversal Jacobi field along $\gamma_{\xi_p}$, and  $s_{1}, s_{2}\in I_\delta$ so that  $X(s_{1})=0=X(s_{2})$, then  
$X(t)=0$ for all $t$, see e.g.  \cite[Lemma 2.1]{PaoloDaniel}. Also note that this $\delta>0$ is the same for other $\gamma_{\xi_x}$ for  $x\in L$  as a consequence of part $(c)$ of Lemma \ref{comparingJh}. 
 Using again that
the critical points of $\pi$ on  $\gamma_{\xi_x}$ are isolated,
 we  can also suppose, reducing $\delta$ if necessary, that the point  $\gamma_{\xi_x}(r)$ is the only critical point of $\pi$ on  $\gamma_{\xi_x}|_{I_\delta}$ for  $x\in L$.  

Let $\tilde{s}\in (r-\delta,r)$ and $\tilde{p}=\gamma_{\xi_p}(\tilde{s})$. Define the unit vector field $\tilde{\xi}$ 
along $L_{\tilde{p}}$ so that $\tilde{\xi}_{\tilde{p}}=\dot{\gamma}_{\xi_p}(\tilde{s})$. 
Recall that $\eta^{\tilde{s}}_{\xi}: L_{p}\to L_{\tilde{p}}$ is a diffeomorphism (see Lemma \ref{SubAnal-lemma3}). 
Set $\tilde{r}:=r-\tilde{s}$. Also note that $\eta^r_{\xi}=\eta^{\tilde{r}}_{\tilde{\xi}}\circ\eta^{\tilde{s}}_{\xi}$. Therefore to prove  
that $x\rightarrow \dim \rank d(\eta^r_{\xi})_{x}$ is locally constant at $p\in L$,   
 it suffices to prove
that  $\tilde x\rightarrow \dim \rank d(\eta^{\tilde{r}}_{\tilde{\xi}})_{\tilde{x}}$ is locally constant at $\tilde{p}=\eta^{\tilde{s}}_{\xi}(p) \in L_{\tilde{p}}$.

We claim that \emph{$L_{\tilde{p}}$-focal points along the horizontal segment of geodesic 
$\gamma_{\tilde\xi_{\tilde{x}}}|_{(0,\tilde{r}+\delta)}$ are of tangential
type.} In other words,  $\gamma_{\tilde\xi_{\tilde{x}}}(t_{1})$ is a focal point  with
multiplicity $k$ ($0<t_{1}<\tilde{r}+\delta$) if and only if $\tilde{x}$ is a critical point of the endpoint map $\eta^{t_{1}}_{\tilde{\xi}}$ 
and  $\dim \ker d \eta^{t_{1}}_{\tilde{\xi}}=k.$ 

In order to prove the claim, consider
an $L_{\tilde{p}}$-Jacobi field  $\tilde{J}$ along $t\to \gamma_{\tilde\xi_{\tilde{x}}}(t)=\gamma_{\xi_x}(t+\tilde{s})$ and 
assume that $\tilde{J}(t_{1})=0$ for  some $0<t_{1}<\tilde{r}+\delta$. Let $J$ be the Jacobi field along $\gamma_{\xi_x}$ so that $\tilde{J}(t)= J(t+\tilde{s})$.  Observe that $J$ is the variation vector field of a variation by orthogonal geodesics to $L_{\tilde p}$ 
 (see Prop. \ref{proposition-jacobifield-variation})  and then its projection is a Jacobi field of $\pi\circ \gamma_{\xi_x}$. By part $(a)$ of  Lemma \ref{comparingJh}, $J^\h$ is a solution of the transversal Jacobi equation. 
As $J^{\h}(\tilde{s})=0=J^{\h}(\tilde{s}+t_{1})$,  from the choice of $I_\delta$  and part $(c)$ of Lemma \ref{comparingJh}, 
we conclude  
that $J^{\h}(t)=0$ for all $t$. Therefore  $\tilde{J}^{\h}(t)=0$ for all $t$. This fact and 
Lemma \ref{lemma-variacao-jacobiholonomia} imply that 
$\gamma_{\tilde\xi_{\tilde{x}}}(t_{1})$ is an $L_{\tilde p}$-focal point 
 if and only if $\tilde{x}$ is a critical point of the endpoint map $\eta^{t_{1}}_{\tilde{\xi}}$.

From what we have discussed above, we have concluded that: 
\begin{eqnarray}
\label{pt-focal-gamma1} 
m (\gamma_{\tilde\xi_{\tilde{x}}})= \dim\ker d\eta^{\tilde{r}}_{\tilde{\xi}}(\tilde{x}),
\end{eqnarray}
where $m(\gamma_{\tilde\xi_{\tilde{x}}})$ denotes the number of focal points on $\gamma_{\tilde\xi_{\tilde{x}}},$ each counted with its multiplicities.    

On the other hand, by continuity of the Morse index we have
\begin{eqnarray}
\label{morse-desigualdade-ptfocais}
m (\gamma_{\tilde\xi_{\tilde{x}}})\geq m(\gamma_{\tilde\xi_{\tilde{p}}}),
\end{eqnarray}
for all $\tilde{x}$ in some neighbourhood of $\tilde{p}$ in $L_{\tilde{p}}.$

Eq.~\eqref{pt-focal-gamma1} and \eqref{morse-desigualdade-ptfocais} together with the elementary inequality 
$\dim\ker d\eta^{\tilde{r}}_{\tilde{\xi}}(\tilde{x})\leq\dim  \ker d\eta^{\tilde{r}}_{\tilde{\xi}}(\tilde{p})$  (as $\dim {\rm Im}\,d\eta^{\tilde{r}}_{\tilde{\xi}}(\tilde{x})\geq\dim  {\rm Im} \,d\eta^{\tilde{r}}_{\tilde{\xi}}(\tilde{p})$) 
imply that  $\dim \ker d\eta^{\tilde{r}}_{\tilde{\xi}}$  is  constant in a neighborhood of $\tilde{p}$ on $L_{\tilde{p}}$.
As  $\eta^{\tilde{s}}_{\xi}: L_{p}\to L_{\tilde{p}}$ is a diffeomorphism, 
 it follows that  $\dim \ker d\eta^r_{\xi}$  is  constant in a neighborhood of $p$ on $L_{p}$,  which concludes. 

 
\end{proof}


\subsection{Finslerian character}
In this section, we will additionally assume that the fibers constitute a singular foliation and will prove  its  transnormality.
 This will follow directly from Lemmas \ref{regularhorizontal}, \ref{SubAnal-lemma2}  and \ref{SubAnal-lemma8}.

\begin{lemma}\label{regularhorizontal} 
	If $\mathcal F_{\pi} $ 	is a singular foliation, $\gamma$ is a geodesic with $\dot\gamma(0)\in \nu (L_{\gamma(0)}) $ and $\gamma(0)$ is a regular point,  then $\gamma$ is horizontal.  
\end{lemma}
\begin{proof} 
 By Lemma~\ref{SubAnal-lemma2}, the geodesic $\gamma$ 
	is orthogonal to the regular leaves associated with regular values of $\pi$. 
 Now suppose  that $\gamma(t_0)$ is a singular point and $u\in T_{\gamma(t_0)}L_{\gamma(t_0)}$. Then as $(M,\F)$ is a singular foliation, there exists a vector field $X$ in a neighborhood of $\gamma(t_0)$ such that $X$ is tangent to the leaves and $X_{\gamma(t_0)}=u$. By Lemma \ref{lemma:V=L}, the singular points along $\gamma$ are isolated, so $g_{\dot\gamma(t)}(X_{\gamma(t)},\dot\gamma(t))=0$ for all $t$ close enough to $t_0$. By continuity, we conclude that $g_{\dot\gamma(t_0)}(u,\dot\gamma(t_0))=0$, as desired. 
	\end{proof}

Note that the above lemma does not directly  imply that the foliation $\F_\pi$ 
is a singular Finsler foliation, because the lemma was only proved for a geodesic $\gamma$ starting at a regular point $p\in L$.

\begin{lemma}
	\label {SubAnal-lemma5} If  $\F_\pi$  is a singular foliation  and $c$ is a regular value of $\pi$,  then
	$\eta^r_{\xi}:\pi^{-1}(c) \to \pi^{-1}(d)$ is surjective even when  $d$ is a singular value.
\end{lemma}
\begin{proof}
	
	The proof consists of two steps. In the first step  we are going to check
	that the endpoint map
	$\eta^r_{\xi}:\pi^{-1}(c)\to\pi^{-1}(d)$ 
	is an open map. 
	And in the second step we have to check that the set $\eta^r_{\xi}(\pi^{-1}(c))$ is a closed subset of $\pi^{-1}(d)$. This ends the proof because the fibers are connected.

	For $q=\eta_{\xi}^r(x_0)\in \pi^{-1}(d)$,  consider a plaque $P_q$ of $\pi^{-1}(d)$ 
	which contains $q$ and a past tubular neighborhood $U=\tilde{\mathcal O}(P_q,\varepsilon)$ of the plaque $P_q$ of radius $\varepsilon>0$. 
	  By Lemma \ref{SubAnal-lemma3}  and the fact that the singular points along $\gamma_{\xi_{x_0}}$ are isolated (see Lemma \ref{lemma:V=L}) 
		 and using Lemma \ref{regularhorizontal}, 
	we can assume without loss of generality that
	$x_0\in U$ and $\gamma_{\xi_{x_0}}$ is the minimizing geodesic from $x_0$ to $P_q$.
  This implies that all the geodesics $\gamma_{\xi_x}$ with  $x\in \tilde U\subset U\cap \pi^{-1}(c)$, being $\tilde U$ a small enough open subset 
  of $\pi^{-1}(c)$ which contains $x_0$, are minimizing geodesics from $x$ to $P_q$, and then 
	\begin{equation}
	\label{eq-SubAnal-lemma5}
	\eta^r_{\xi}=\pi^{\nu}|_{\tilde U},
	\end{equation}
	where $\pi^{\nu}$ is the (past)  footpoint projection.
	Since $\F_\pi$ is a singular foliation, the map $\pi^{\nu}|_{\tilde U}$  is open by \cite[Lemma 3.10]{Alexandrino-Alves-Javaloyes}.
	This fact and Eq. \eqref{eq-SubAnal-lemma5} imply that  the endpoint map
	$\eta_{\xi}$ is an open map.

	Consider a sequence $\{y_n\}\subset \eta^r_{\xi}(\pi^{-1}(c))$ that converges to a point $ y \in \pi^{-1}(d)$.
	Let $\{x_n\} \subset \pi^{-1}(c)$ be a sequence so that $\eta^r_{\xi}(x_n)=y_n$. Since, for $n$ big enough,  $x_n\in \overline{B^-( y,2r)}$, 
	 denoting by $B^-( y,2r)=\{x\in M: d_F(x,y)<2r\}$, namely, the backward ball or radius $2r$ centered at $y$, 
	we have a  convergent  subsequence
	$x_{n_i}\to x\in \overline{B^-( y,2r)}.$ Therefore $\eta^r_{\xi}(x)=\lim \eta^r_{\xi}(x_{n_i})=\lim y_{n_i} =	y$ and hence 
	$y\in  \eta^r_{\xi}(\pi^{-1}(c)).$
	
\end{proof}

\begin{lemma}
\label{SubAnal-lemma6}
 Let $L_{p}$ be a regular leaf, and $L_{q}$ be a singular leaf. Then they are parallel, 
i.e., if $y_{0}$ and $y_{1}$ are points of 
 $L_q$, then $d(L_{p}, y_{0})=d(L_{p},y_{1})$ and $d( y_{0},L_{p})=d(y_{1},L_{p})$. 
\end{lemma}
\begin{proof}
For $i=0,1$, let $\gamma_i:[0,r_{i}]\to M$ be a minimizing segment of geodesic joining $L_{p}$ with $y_i$, i.e.,
$\gamma_{i}(0)\in L_{p}$, $\gamma_{i}(r_{i})=y_{i}$ and 
$d(L_{p}, y_{i})=\int_{0}^{r_{i}}  F(\dot{\gamma}_{i}) dt =: \ell(\gamma_{i})$.
Since $\eta^{r_1}_{\xi_{1}}$ is surjective by Lemma \ref{SubAnal-lemma5}, we can   transport $\gamma_1$    and get a  (maybe non-unique)  new curve $|| \gamma_{1}$ joining $L_{p}$ to $y_{0}$
such that $\ell(|| \gamma_{1})=\ell(\gamma_{1})$. Since  $\gamma_{0}$ is minimal, we conclude that $\ell(\gamma_{0})\leq \ell(\gamma_{1})$. An analogous argument
allows us to conclude that  $\ell(\gamma_{1})\leq \ell(\gamma_{0})$ and hence,   $\ell(\gamma_{1})= \ell(\gamma_{0})$ as required,  which implies that  $d(L_{p}, y_{0})=d(L_{p},y_{1})$. The other equality is analogous. 
\end{proof}

\newcommand{\Uc}{\ensuremath{\mathcal{U}}}
\begin{lemma}
\label{SubAnal-lemma7}
 Given a point $\tilde q$ in a singular leaf $L_{\tilde q}$, there exists a plaque $P_{\tilde{q}}$ of  $\tilde q$
admitting a future (resp. past) tubular neighborhood 
such that
 if  $P_{q}$ is a (regular or singular) 
plaque in this neighborhood, 
then $P_{q}\subset \mathcal{C}_{r}^{+}(P_{\tilde{q}})$  (resp. $P_{q}\subset \mathcal{C}_{r}^{-}(P_{\tilde{q}})$)
for  an appropriate $r$. 
\end{lemma}
\begin{proof}
 We will consider only the future case, as  the past one can be recovered by using the reverse Finsler metric $\tilde F(v)=F(-v)$. 
 Let $\Uc$ be a totally convex neighborhood of $\tilde q$, namely, given two points $p_1,p_2\in \Uc$ there exists a unique geodesic joining $p_1$ to
$p_2$ and this geodesic is minimizing (see \cite{W1,W2} for their existence). Now consider a plaque $\tilde P_{\tilde q}$ of $\tilde q$, a future tubular neighborhood
$\tilde U \subset \Uc$, a smaller plaque $P_{\tilde q}$ of $\tilde q$ and a future tubular neighborhood $U\subset \tilde U$ of $P_{\tilde q}$ such that if $D>0$ is the diameter of 
$ U$, then $B^+(p,2D)\cup B^-(p,2D)\subset \tilde U$ for every $p\in U$. 
Let $q\in U$, $P_q$ the plaque of $q\in U$, and $y_{0}, y_{1}\in P_q$. Our goal is to check that   $d(P_{\tilde{q}}, y_{0})=d(P_{\tilde{q}},y_{1})$.
This will imply $P_{q}\subset \mathcal{C}_{r}^{+}(P_{\tilde{q}})$. 

 For $i=0,1$, let $\gamma_i:[0,r_{i}]\to M$ be a minimal segment of geodesic joining 
$P_{\tilde{q}}$ with $y_i$, in particular $d(P_{\tilde{q}}, y_{i})=\ell(\gamma_{i})$, for $i=0, 1$.
 Choose an arbitrary regular point $p\in M$. Then there exists a minimizing geodesic from $L_p$ to $\tilde q$, whose singular points are isolated. 
 As a consequence, we can   
consider a sequence of regular points $p_{m}$  of the future tubular neighborhood $U$ 
converging to $\tilde{q}$. Moreover, consider a sequence of
minimal segments of geodesic $\beta_{i}^{m}$ joining  $\tilde P_{p_{m}}$  to $\gamma_{i}(0)$ (even if $\tilde P_{p_m}$ is not necessarily closed, the 
minimizing geodesic $\beta^m_i$ exists because $p_m\in \bar B^-(\gamma_i(0),D)\subset \tilde U$ by the choice of $U$). In particular,
 $d(\tilde P_{p_{m}}, \gamma_{i}(0))=\ell(\beta^{m}_{i})$.
Finally, consider   minimal segments of geodesic $\gamma_{i}^{m}:[0,r_m]\rightarrow M$ joining $\tilde P_{p_{m}}$ to $y_{i}$, which there exist by the same 
reason as the $\beta^m_i$'s. 
It follows that $\ell(\gamma_{i}^{m})=d(\tilde P_{p_{m}},y_{i})$ and $\ell(\gamma_{i}^{m})\leq \ell(\gamma_{i}\star \beta_{m}^{i})$,  where $\star$ stands for the concatenation of the two curves. 
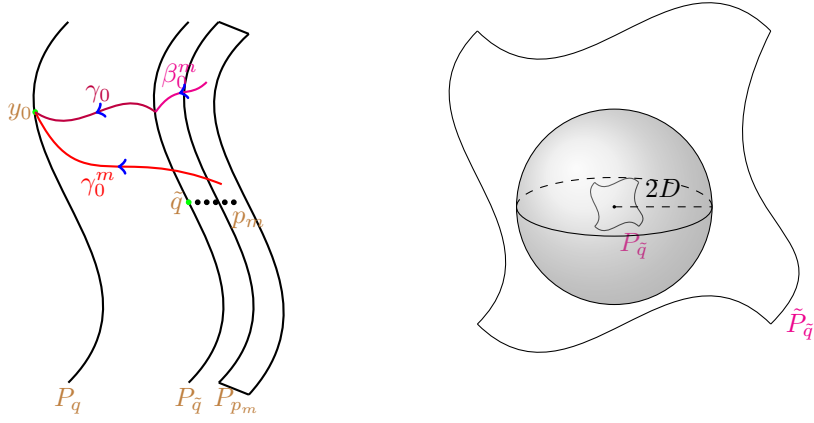
\begin{figure}
\begin{tikzpicture}[scale=0.8]
\begin{scope}[shift={(-3.5,0)}]
  \draw[thick] (0,0) .. controls (2,2) and (-2,4) .. (0,6);
  \draw[thick] (0,6) -- (0.5,5.8);
   \draw[thick] (0.5,-0.2) .. controls (2.5,2) and (-1.5,4) .. (0.5,5.8);
    \draw[thick] (0,0) -- (0.5,-0.2);
     \draw[thick] (-0.5,0) .. controls (1.5,2) and (-2.5,4) .. (-0.5,6);
      \draw[thick] (-2.5,0) .. controls (-0.5,2) and (-4.5,4) .. (-2.5,6);
       \node[brown] at  (-2.5,-0.3)   {$P_q$};
        \node[brown] at  (-0.5,-0.3)   {$P_{\tilde q}$};
         \node[brown] at  (0.3,-0.3)   {$P_{p_m}$};
         \filldraw[thick,green] (-0.5,3) circle(0.03);
          \node[brown] at  (-0.7,3)   {$\tilde q$};
           \filldraw[thick] (-0.35,3) circle(0.03);
           \filldraw[thick] (-0.2,3) circle(0.03);
           \filldraw[thick] (-0.05,3) circle(0.03);
           \filldraw[thick] (0.1,3) circle(0.03);
            \filldraw[thick] (0.25,3) circle(0.03);
              \node[brown] at  (0.5,2.7)   {$p_m$};
               \filldraw[thick,green] (-3.05,4.5) circle(0.03);
                \node[brown] at  (-3.3,4.5)  {$y_0$};    
                 \node[purple] at  (-2,4.8)   {$\gamma_0$};        
                 \node[red] at  (-2,3.3)   {$\gamma^m_0$};               
                  \node[magenta] at  (-0.65,5.1)   {$\beta^m_0$};

\begin{scope}[decoration={markings,
    mark=at position 0.5 with {\arrow[blue,very thick]{>}}}
              ]
 \draw[postaction={decorate},thick,red] (0.05,3.3)  .. controls (-1.7,4) and (-2.3,3)  .. (-3.05,4.5);
  \draw[postaction={decorate},thick,purple] (-1.05,4.5)  .. controls  (-1.7,5) and (-2.3,4) .. (-3.05,4.5);
   \draw[postaction={decorate},thick,magenta] (-0.2,5) .. controls (-0.5,4.7) and (-0.7,5) ..  (-1.05,4.5);
\end{scope}
\end{scope}
\end{tikzpicture}
\begin{tikzpicture}[scale=1.3]
\draw (0,0) .. controls (1,-1) and (2,1) .. (3,0); 
\draw (3,0) .. controls (4,1) and (2,1) .. (3,3);
\draw (0,0) .. controls (1,1) and (-1,2) .. (0,3);
\draw (0,3) .. controls (1,2) and (2,4) .. (3,3); 
\begin{scope}[shift={(1.2,1)},scale=0.15]
\draw (0,0) .. controls (1,-1) and (2,1) .. (3,0); 
\draw (3,0) .. controls (4,1) and (2,1) .. (3,3);
\draw (0,0) .. controls (1,1) and (-1,2) .. (0,3);
\draw (0,3) .. controls (1,2) and (2,4) .. (3,3); 
\end{scope}
 \node[magenta] at  (1.6,0.8)   {$P_{\tilde q}$};
\begin{scope}[shift={(1.4,1.2)},scale=0.5]
\shade[ball color = gray!40, opacity = 0.4] (0,0) circle (2cm);
  \draw (0,0) circle (2cm);
  \draw (-2,0) arc (180:360:2 and 0.6);
  \draw[dashed] (2,0) arc (0:180:2 and 0.6);
  \fill[fill=black] (0,0) circle (1pt);
  \draw[dashed] (0,0 ) -- node[above]{$2D$} (2,0);
\end{scope}
 \node[magenta] at  (3.3,0)   {$\tilde P_{\tilde q}$};
\end{tikzpicture}
\caption{\label{figure1} The diagram to the left represents the different curves used in the proof of Lemma \ref{SubAnal-lemma7}. The diagram to the right depicts both plaques $P_{\tilde q}$ and $\tilde P_{\tilde P}$. The minimizing curves inside the tubular neighborhood $U$ exist as short curves cannot reach the boundary, because of the hypothesis on the diameter.}
\end{figure}

 Observe that the image of the curves $\gamma_i^m$ is contained in $\bar B^-(y_i,D)$.  As a consequence, there exists a convergent subsequence
$\gamma_i^{n_k}(0)$ converging to some $\hat p\in \tilde P_{\tilde q}\subset \tilde U$ (recall that the plaques $\tilde P_{p_m}$ in $U$ must necessarily converge to $\tilde P_{\tilde q}$). As all the sequence and the limit lie in the convex neighborhood $\Uc$, 
we can guarantee that 
$\gamma_{i}^{n_k}$  converges to a segment of geodesic $\widehat{\gamma}_{i}$ joining $\tilde P_{\tilde{q}}$ to $y_{i}$. 
By Lemma \ref{SubAnal-lemma6},  $\ell(\beta_{i}^{m})\to 0$ when $m\to \infty$, and hence  we conclude that
$\ell(\hat{\gamma}_{i})\leq \ell(\gamma_{i})$.
Since $ d(\tilde P_{\tilde{q}}, y_{i}) =d(P_{\tilde{q}}, y_{i})=\ell(\gamma_{i})$, we infer that
$\ell(\hat{\gamma}_{i}) = \ell(\gamma_{i})$.

On the other hand,
 as the points $p_m$ are regular, one can  apply the same techniques of parallel transport of the proof of Lemma \ref{SubAnal-lemma6} to prove that \ $\ell(\gamma_{0}^{m})=\ell(\gamma_{1}^{m})$ and hence
$\ell(\hat{\gamma}_{1})=\ell(\hat{\gamma}_{0})$.
Therefore $\ell(\gamma_{0})=\ell(\gamma_{1})$ and this concludes the proof of the lemma.


\end{proof}

\begin{lemma}
	\label{SubAnal-lemma8}
	Let $x$ be a point of a singular leaf $L_{\tilde{q}}$. Then the geodesic $\gamma_{\xi_x}$ is horizontal, i.e., 
	orthogonal to the leaves of the singular foliation $\F=\{\pi^{-1}(c)\}$.
	
	\end{lemma} 
\begin{proof}
The proof follows directly from Lemmas \ref{SubAnal-lemma7}  and \ref{lemma-equidistant}.  


\end{proof}

%
\section{On Jacobi fields and curvature in Finsler Geometry}

\label{apendice}

\subsection{Anisotropic connection and curvature}
\label{apendice1}
In this section we review a few facts about anisotropic connections, see  \cite{Miguelanisotropic,J20}.

As the tensors associated with Finsler metrics depend on the direction $v\in TM\setminus \bf 0$, it is necessary a connection that take into account this fact, namely, it depends on the direction. In the following, we will denote by ${\mathfrak X}(M)$ the space of smooth vector fields on $M$ and by ${\mathcal F}(M)$, the subset of smooth real functions on $M$.
\begin{definition}\label{aniconnection}
	An  {\em anisotropic connection} is  a map
	\[\nabla: A\times \mathfrak{X}(M)\times\mathfrak{X}(M)\rightarrow TM,\quad\quad (v,X,Y)\mapsto\nabla^v_XY
	\in T_pM\setminus 0,\quad \text{with $v\in T_pM$, }\]
	where $A=TM\setminus \bf 0$, such that for any $X,Y\in  {\mathfrak X}(M)$,  the map $A\ni v\rightarrow \nabla^v_XY\in TM$ is smooth,
	and  for all $p\in M$ and $v\in T_pM\setminus \{0\}$,
	\begin{enumerate}[(i)]
		\item $\nabla^v_X(Y+Z)=\nabla^v_XY+\nabla^v_XZ$, for any $X,Y,Z\in  {\mathfrak X}(M)$,
		\item $\nabla^v_X(fY)=X_p(f) Y_{p}+f(p) \nabla^v_XY $ for any $f\in {\mathcal F}(M)$ and $X,Y\in  {\mathfrak X}(M)$,
		\item $\nabla^v_{fX+hY}Z=f(p)\nabla^v_XZ+h(p) \nabla^v_YZ$, for any $f,h\in {\mathcal F}(M)$ and $X,Y,Z\in  {\mathfrak X}(M)$.
	\end{enumerate}
\end{definition}
Observe that the properties $(ii)$ and $(iii)$ imply, respectively, that $\nabla^v_XY$ depends only on the value of $Y$ on a neighborhood of $p$, and on the value of $X_p$. So sometimes we will put only $X_p$ and it will make sense to compute $\nabla^v_XY$ in a chart.
If $(\Omega,\varphi=(x^1,\ldots,x^n))$ is a chart of $M$, with its associated partial vector fields denoted by $\partial_1,\ldots,\partial_n$, the {\em Christoffel symbols} of $\nabla$ are the functions $\Gamma^k_{\,\, ij}: T\Omega\setminus {\bf 0}\rightarrow \RR$, $i,j,k=1,\ldots,n$, determined by
\begin{equation}\label{chrissymbols}
 \nabla^v_{\partial_i}\partial_j=\sum_{k=1}^n\Gamma^k_{\,\, ij}(v) \partial_k,
 \end{equation}
for $i,j=1,\ldots,n$.
 Given a vector field $V$ without singularities  on an open set $\Omega\subset M$, the anisotropic connection $\nabla$ induces an affine connection $\nabla^V$ on $\Omega$ defined as $(\nabla^V_XY)_p=\nabla^{V_p}_{\tilde X}\tilde Y$ for any $X,Y\in {\mathfrak X}(\Omega)$, where $\tilde X,\tilde Y\in {\mathfrak X}(M)$ are extensions of $X$ and $Y$, respectively, namely, they coincide with $X$ and $Y$ in a neighborhood of $p\in\Omega$. Moreover, if we define the \emph{vertical derivative of $\nabla$} as the $(1,3)$- anisotropic tensor $P$ given by
 \[ P_v(X,Y,Z)=\frac{\partial}{\partial t}\left(\nabla^{v+tZ_{p}}_XY\right)|_{t=0},\]
 for $p\in M$, $v\in T_pM\setminus\{0\}$ and $X,Y,Z\in {\mathfrak X}(M)$, then it is possible to define the curvature tensor of the anisotropic connection as
 \begin{equation}\label{curten}
 R_v(X,Y)Z=(R^V)_p(X,Y)Z-(P_V)_p(Y,Z,\nabla^V_XV)+(P_V)_p(X,Z,\nabla^V_YV),
 \end{equation}
 where $V\in {\mathfrak X}(\Omega)$ is an extension of $v\in T_pM\setminus\bf 0$,  $X,Y,Z\in {\mathfrak X}(M)$, $P_V$ is the tensor such that $(P_V)_p=P_{V_p}$ and $R^V$ is the curvature tensor of the affine connection $\nabla^V$. It is important to observe that the curvature tensor $R_v$ does not depend on the vector fields $V,X,Y,Z$ chosen to make the computation, but only on $v,X_p,Y_p$ and $Z_p$, \cite[Prop. 2.5]{J20}. With an anisotropic connection at hand, we can also compute the tensor derivative of any anisotropic tensor (see \cite[\S 3]{Miguelanisotropic}), and this computation can be done using $\nabla^V$ (see \cite[Remark 15]{Miguelanisotropic}). In most of the references of Finsler Geometry, they use a connection along the vertical bundle of $TM$. For a relationship between both types of connections see \cite[\S 4.4]{Miguelanisotropic}.


\subsection{A few proofs about Finsler Jacobi fields}
\label{apendice2}

Here, for the sake of completeness,
 we present the proof of a few results about Jacobi fields that we have used. 


\subsubsection{Proof of Lemma \ref{lemma-W-selfadjoint}}\label{lema34}
%


Consider $J_{1}, J_{2}\in \selfadjointW$.   Using the shape operator $\mathcal{S}_{\dot{\gamma}(t_0)}$ (see Definition \ref{definition-eq-L-Jacobi}), it follows  that
\begin{eqnarray*}
g_{\dot{\gamma}(t_0)}(J_{1}'(t_{0}), J_{2}(t_{0}))  &= & g_{\dot{\gamma}(t_0)}(\mathcal{S}_{\dot{\gamma}(t_0)}(J_{1}(t_{0})), J_{2}(t_{0})) 
                                           \\ &= & g_{\dot{\gamma}(t_0)}(J_{1}(t_{0}), \mathcal{S}_{\dot{\gamma}(t_0)} (J_{2}(t_{0}))) 
																					 =  g_{\dot{\gamma}(t_0)}(J_{1}(t_{0}), J_{2}'(t_{0})),
\end{eqnarray*}
 using that $\mathcal{S}_{\dot{\gamma}(t_0)}$ is self-adjoint (see \cite[Prop. 3.5]{JavSoa15}). 
As we saw in Remark \ref{remark-selfadjoint-at-zero} this  implies that $\selfadjointW$ is self-adjoint for all $t$.

 In order to see that the dimension of $\selfadjointW$ is equal to $n-1$, observe first that the dimension of the $L$-Jacobi fields is $n$.  Recall that for every $t\in I$, if we define $\dot\gamma(t)^\perp:=\{ v\in T_{\dot\gamma(t)} M: g_{\dot\gamma(t)}(\dot\gamma(t),v)=0\}$, we have the splitting
\begin{equation}\label{splitgamma}
T_{\gamma(t)}M={\rm span}\{\dot\gamma(t)\}+\dot\gamma(t)^\perp,
\end{equation}
  and we will denote by ${\rm tan}_{\gamma}:{\mathfrak X}(\gamma)\rightarrow {\mathfrak X}(\gamma)$ and ${\rm nor}_{\gamma} :{\mathfrak X}(\gamma)\rightarrow {\mathfrak X}(\gamma)$ the first and second projection of the splitting \eqref{splitgamma} at
  every $t\in I$. Observe that  given an arbitrary Jacobi field $J$, then $J^T={\rm tan}_{\gamma}J$ and $J^\perp={\rm nor}_{\gamma} J$ are again Jacobi fields along $\gamma$ (see for example \cite[Lemma 3.17]{JavSoa15}). Moreover, as $J$ is $L$-Jacobi, then $J(t_0)$ is tangent to $L$ and therefore $g_{\dot\gamma(t_0)}$-orthogonal to $\dot\gamma(t_0)$. As a consequence, $J^T(t_0)=0$. Taking into account that a Jacobi field is tangent to $\gamma$ if and only if $J(t)=(a_1 t+a_2) \dot\gamma(t)$ (see \cite[Lemma 3.17 (i)]{JavSoa15}), we conclude that $J^T(t)=a_1 t\dot\gamma(t)$, which is an $L$-Jacobi field and then $J^\perp$ is also an $L$-Jacobi field, because $J^\perp=J-J^T$. Let ${\mathcal J}_L(\gamma)$  be the space of $L$-Jacobi fields along $\gamma$ and ${\mathcal J}_L^T(\gamma)=\{J\in {\mathcal J}_L(\gamma):J^\perp=0\}$, ${\mathcal J}_L^\perp(\gamma)=\{J\in {\mathcal J}_L(\gamma):J^T=0\}$. Therefore, we have a map 
\[ \phi:{\mathcal J}_L^T(\gamma)\times {\mathcal J}_L^\perp(\gamma)\rightarrow {\mathcal J}_L(\gamma), (J_1,J_2)\rightarrow J_1+J_2, \]
which is well-defined and one-to-one.  
  As we have seen above that ${\mathcal J}_L^T(\gamma)$ has dimension one, it follows that ${\mathcal J}_L^\perp(\gamma)$ has dimension $n-1$, which concludes. 


\subsubsection{Proof of Proposition \ref{proposition-jacobifield-variation}}\label{prop35}


 Observe that  the dimension of the submanifold $TL^\perp\subset TM$ of orthogonal vectors to $L$ is $n=\dim M$ (see for example \cite[Lemma 3.3]{JavSoa15}). Now consider a curve $(-\varepsilon,\varepsilon)\ni s\rightarrow N(s)\in TL^\perp$, with $\beta(s)=\rho(N(s))$ a curve in $L$ such that $N(0)=\dot\gamma(t_0)$ (where  $\rho:TM\rightarrow M$  is the natural projection). We can construct a variation $\Lambda(t,s)=\gamma_{N(s)}(t-t_0)$, $\Lambda:I\times (-\varepsilon,\varepsilon)\rightarrow M$, which is given by the $L$-orthogonal geodesics $\gamma_{N(s)}$ whose velocity at $0$ is $N(s)$. Observe that $\Lambda(t_0,s)=\beta(s)$, $\frac{\partial\Lambda}{\partial t}(t_0,s)=N(s)$, $\frac{\partial\Lambda}{\partial s}(t_0,s)=\dot\beta(s)$.  It is not difficult to show that $V(t)=\frac{\partial\Lambda}{\partial s}(t,0)$ is an $L$-Jacobi field along $\gamma$.   Indeed, it is Jacobi, because its longitudinal curves are geodesics (see \cite[Prop. 3.13]{JavSoa15}). In order to show that it is $L$-Jacobi, observe that $V(t_0)=\dot\beta(0)$ and if $(\Omega,\varphi)$ is a chart in a neighborhood of $\gamma(t_0)$, with $N^i,\beta^i,\gamma^i$, $i=1,\ldots,n$, the coordinates of $N$, $\beta$ and $\gamma$, and $\Gamma^i_{jk}:T\Omega\setminus 0\rightarrow \mathbb{R}$ the Christoffel symbols of the Chern connection in $(\Omega,\varphi)$, then
 \begin{equation}\label{coordinDgamma}
 V'(t_0)=D^{\dot\gamma}_{\gamma}V(t_0)=D^{\dot\gamma}_\beta N (0)= (\dot N^i(0)+\dot\beta^j(0) \dot\gamma^k(t_0) \Gamma^i_{jk}(\dot\gamma(t_0)))\partial_i,
 \end{equation} 
 (using \cite[Eq. (7)]{JavSoa15} in the second equality).
Therefore $V(t_0)$ is tangent to $L$ and the second equality of \eqref{coordinDgamma} implies that  ${\rm tan}^L_{\dot\gamma(t_0)} V'(t_0)=\mathcal{S}_{\dot{\gamma}(t_0)}(V(t_0))$, which concludes that $V$ is $L$-Jacobi. In order to see that all the $L$-Jacobi fields can be obtained with variations of $L$-orthogonal geodesics,
observe that the  formula \eqref{coordinDgamma} allows us to define a map  \[\rho:T_{\dot\gamma(t_0)}TL^\perp\rightarrow T_{\gamma(t_0)}M\times T_{\gamma(t_0)}M,\] 
defined as $\rho(\dot N(0))= (V(t_0),V'(t_0))$. It is straightforward to check, using \eqref{coordinDgamma}, that $\rho$ is injective.  Indeed, this follows from the expression of the coordinates of $\frac{dN}{ds}(0)$ in a natural chart of $TTM$ associated with $(\Omega,\varphi)$, which are $\gamma^1(0),\ldots, \gamma^n(0),$ $\dot\gamma^1(t_0),\ldots, \dot\gamma^n(t_0),\dot\beta^1(0),\ldots,\dot\beta^n(0),\dot N^1(0),\ldots, \dot N^n(0)$,  taking into account that, by definition, $N(0)=\dot\gamma(t_0)$ and $V(t_0)=\dot\beta(0)$. Moreover, as we have seen above, the image of $\rho$ is contained in the initial values for the Cauchy problem of $L$-Jacobi fields. As the dimension of $T_{\dot\gamma(t_0)}TL^\perp$ coincides with the dimension of the space of $L$-Jacobi fields (in both cases, the dimension is $n$), this easily implies that every $L$-Jacobi field can be realized as the variational vector field of a variation by $L$-orthogonal geodesics.


\subsubsection{Proof of Lemma \ref{lemma-variacao-jacobiholonomia} }\label{lema36}
%
 The fact that (b) implies (a) follows from  Proposition \ref{proposition-duran-geodesics-submersion}, because all the geodesics 
in the variation $\psi$ project into the same geodesic on $B$, and as a consequence $J(t)=\frac{\partial}{\partial s} \psi(t,0)$ is vertical for every $t\in I$. 
Let us check that (a) implies (b). First note that $J$
is  determined by its first initial condition $J(t_{0})$, because $J$ is vertical.   In fact 
assume by contradiction that there exists another vertical $L$-Jacobi field  $\tilde{J}$
so that $\tilde{J}(t_{0})=J(t_{0})$ and $J\neq \tilde{J}$. Set $\hat{J}:=J - \tilde{J}$. Hence $\hat{J}$ is a vertical
Jacobi field   with $\hat{J}(t_{0})=0$, and then $\hat{J}'(t_{0})=\lim_{t\rightarrow t_0} \frac{1}{t-t_0}\hat{J}(t)$, which implies
that $\hat{J}'(t_{0})$ is also tangent to $L$.  Being $\hat J$ $L$-Jacobi (recall 
 Def. \ref{definition-eq-L-Jacobi}), it follows that the tangent part to $L$ of  $\hat{J}'(t_{0})$ is 
zero, which concludes that $\hat{J}=0$, as both initial conditions are zero.  
Now consider   a curve $\beta:(-\varepsilon,\varepsilon)\rightarrow L$ such that $\dot\beta(0)=J(t_0)$ and the Jacobi field $\tilde{J}(t)= \frac{\partial}{\partial s} \psi(t,0)$, being $\psi$ the variation defined in part (b). The fact that (a) implies (b)   follows
from the above discussion since $J(t_0)=\tilde{J}(t_0)=\frac{\partial}{\partial s} \psi(t,0)$, and then $J=\tilde{J}$.




\end{document}